\documentclass[12pt,reqno]{amsart}
\usepackage{mathtools}
\usepackage{amssymb}
\usepackage{mathrsfs}
\usepackage{mathptmx}
\usepackage{enumerate}
\usepackage{stackrel}
\usepackage{upgreek}
\usepackage{xcolor,soul}
\usepackage[normalem]{ulem}
\usepackage{fontawesome}
\usepackage{pdfsync}

\usepackage[text={16cm,20.5cm},centering,asymmetric]{geometry}
\numberwithin{equation}{section}

\newtheorem{theorem}{Theorem}[section]
\newtheorem{lemma}[theorem]{Lemma}
\newtheorem{proposition}[theorem]{Proposition}

\newtheorem{question}{Question}

\theoremstyle{definition}

\newtheorem{definition}[theorem]{Definition}

\newtheorem{example}[theorem]{Example}

\theoremstyle{remark}
\newtheorem{remark}[theorem]{Remark}

\author[P. Gumenyuk]{Pavel Gumenyuk}
\address{P. Gumenyuk: Department of Mathematics, Politecnico di Milano, via E. Bonardi 9, 20133 Milan, Italy.}
\email{pavel.gumenyuk@polimi.it}

\author[M. Kourou]{Maria Kourou$^{\S}$}
\thanks{$^{\S}$Partially supported by the Alexander von Humboldt Foundation.}
\address{M. Kourou: Department of Mathematics, Julius-Maximilians University of W\"urzburg, Emil Fischer Strasse 40, 97074 W\"urzburg, Germany.}
\email{maria.kourou@mathematik.uni-wuerzburg.de}

\author[O. Roth]{Oliver Roth}
\address{O. Roth: Department of Mathematics, Julius-Maximilians University of W\"urzburg, Emil Fischer Strasse 40, 97074 W\"urzburg, Germany.} \email{roth@mathematik.uni-wuerzburg.de}

\title[The angular derivative problem for petals of one-parameter semigroups]{The angular derivative problem\\[1mm] for petals of one-parameter semigroups in the unit disk}


\date{\today}

\newcommand{\UH}{\mathbb{H}}

\newcommand{\R}{\mathbb R}

\newcommand{\C}{\mathbb C}
\newcommand{\Hol}{{\sf Hol}}
\newcommand{\Aut}{{\sf Aut}}
\newcommand{\D}{\mathbb D}

\newcommand{\dist}{\mathop{\mathsf{dist}}}

\newcommand{\id}{\mathop{\mathsf{id}}}

\let\Re=\undefined
\let\Im=\undefined
\DeclareMathOperator{\Re}{\mathrm{Re}}
\DeclareMathOperator{\Im}{\mathrm{Im}}

\newcommand{\Maponto}
{\xrightarrow{\hbox{\lower.2ex\hbox{$\scriptstyle \smash{\mathsf{onto}}$}}\,}}
\newcommand{\Mapinto}
{\xrightarrow{\hbox{\lower.2ex\hbox{$\scriptstyle \smash{\mathsf{into}}$}}\,}}
\newcommand{\anglim}{\angle\lim}

\newcommand{\Real}{\mathbb{R}}
\newcommand{\Natural}{\mathbb{N}}
\newcommand{\Complex}{\mathbb{C}}
\newcommand{\ComplexE}{\overline{\mathbb{C}}}
\newcommand{\UD}{\mathbb{D}}

\newcommand{\diam}{\mathop{\mathsf{diam}}}

\newcommand{\UC}{{\partial\UD}}

\renewcommand{\le}{\leqslant}
\renewcommand{\ge}{\geqslant}
\renewcommand{\leq}{\leqslant}

\newcommand{\confr}[2]{\mathcal R(#1,#2)}
\newcommand{\Gf}{\mathrm{G}}
\newcommand{\Hm}{{\omega\hskip.0125em}}
\newcommand{\St}{{\mathbb S}}
\newcommand{\hdist}{\mathrm{d}}
\newcommand{\Hd}{\uplambda}
\newcommand{\inter}[1]{{#1}^\circ}
\newcommand{\Arg}{\mathop{\mathrm{Arg}}}
\newcommand{\bfit}[1]{{\fontseries{bx}\fontshape{it}\selectfont #1}}
\newcommand{\proofof}[1]{\bfit{Proof of #1}}

\newcommand{\di}{\,\mathrm{d}\hskip.125ex}

\begin{document}
\begin{abstract}
We study the angular derivative problem for petals of one-parameter semigroups of holomorphic self-maps of the unit disk. For hyperbolic petals we prove a necessary and sufficient condition for the conformality of the petal in terms of the intrinsic hyperbolic geometry of the petal and the backward dynamics of the semigroup. For parabolic petals we characterize conformality of the petal in terms of the asymptotic behaviour of the Koenigs function at the Denjoy\,--\,Wolff point.
\end{abstract}

\dedicatory{Dedicated to the memory of Professor Larry Zalcman}

\maketitle

\section{Introduction}
One-parameter semigroups of holomorphic self-maps of the unit disk $\UD:=\{z \in \C \, : \, |z|<1\}$
have been studied for more than a century, but many key properties regarding their boundary and asymptotic  behaviour have been established only in the last two decades.  Apart from being interesting in their own right, one-parameter semigroups in the unit disk act as pivotal role models for  other complex dynamical systems involving holomorphic functions of one or several complex variables --- in the continuous as well as the discrete setting.
 A comprehensive overview  of one-parameter semigroups in the unit disk spanning from the basic theory to the numerous recent achievements can be found in the monograph~\cite{BCD-Book}.

Traditionally, understanding the forward dynamics of the orbits of a dynamical system has been of particular interest. However, in recent years the study of the \textit{backward} dynamics of one-parameter semigroups, see e.g. \cite{ Bracci_et_al2019,ShoiElinZalc,MK-KZ2022}, as well of discrete complex dynamical systems (in one and several complex variables), see e.g. \cite{AR2011,AG2019,Bracci:backward-it,Poggi-Corradini:2003}, has become another focal point of research.

One of the most striking results about the forward dynamics of one-parameter semigroups of the unit disk $\D$ is the continuous version of the celebrated Denjoy\,--\,Wolff Theorem. It  guarantees that all forward orbits of the semigroup converge to the same point $\tau \in \UD\cup\partial\UD$, the Denjoy\,--\,Wolff point of the semigroup or \textit{DW-point}, for short. In contrast, the backward flow of a one-parameter semigroup is defined only on a proper subset $\mathcal W$ of $\D$, the so-called \textit{backward invariant set}. The connected components of the interior of~$\mathcal W$ are simply connected domains and are referred to as the \textit{petals} of the semigroup.
Each petal $\Delta$ is associated to a specific boundary point $\sigma \in \partial \D$. This boundary point $\sigma$  plays a similar role for the backward dynamics of the semigroup as the DW-point plays for the forward dynamics:
$\sigma$~is~the common limit of all backward orbits starting inside the petal $\Delta$. Bearing in mind the conventional terminology in dynamical systems (see e.g. \cite[Chapter IV, Section 17, p. 225]{Amann_1990}),
we call~$\sigma$ the \textit{$\alpha$-point of the petal~$\Delta$}.
We refer to Section~\ref{SEC:Preliminaries} for the precise definition of these concepts, and in particular to the pioneering papers~\cite{Bracci_et_al2019,Analytic-flows,ShoiElinZalc}.

The goal of this paper is to study the angular derivative problem for the petals $\Delta$ of
one-parameter semigroups at their $\alpha$-points $\sigma$. The corresponding angular derivative problem related to the forward dynamics has recently been studied in \cite{Bet2016,BeCDM,CDP2010,NK2019}.

The angular derivative problem for general simply connected domains $G$ is concerned with the existence of a finite non-zero angular derivative of a Riemann map $f$ from the unit disk $\D$ onto $G$ at a given boundary point ${\xi \in \partial \D}$. It can be shown that the existence of the angular derivative of $f$ does not depend on the choice of the conformal map $f$ but only on the domain~$G$ itself. Accordingly, when  this angular derivative condition is satisfied at a boundary point $\xi$, we say that the domain~$G$ is conformal at $f(\xi)$.  In his thesis \cite{Ahlfors1930}, Ahlfors raised the question of finding necessary and sufficient geometric conditions on $G$ near the boundary point $\xi$, so that $G$ is conformal at $f(\xi)$\footnote{"Welche geometrischen Eigenschaften sollen das Gebiet $G$ kennzeichnen, um dass die so bestimmte Funktion im Punkt $\xi$ eine von Null und Unendlich verschiedene Winkelableitung besitze?" \cite[p.~47]{Ahlfors1930}}. This is the angular derivative problem. It has attracted much interest, and continues to do so.
Seminal contributions have been given by Rodin and Warschawski \cite{RW-FinnishJ, RodinWarsch},
Jenkins and Oikawa \cite{JO1977}, followed by many others, e.g. Burdzy~\cite{Burdzy1986}, Carroll~\cite{Carroll1988}, as well as very recently by Betsakos and Karamanlis~\cite{BK2022}.

We can now state our main result. It is concerned  with \textit{hyperbolic} petals, i.e.~petals $\Delta$ for which the $\alpha$-point $\sigma$ is different from the Denjoy\,--\,Wolff point. The simpler case of \textit{parabolic} petals, i.e.~petals for which the $\alpha$-point coincides with the Denjoy\,--\,Wolff point, will be discussed in Section \ref{SEC:parabolic_petals}. We denote the density of the hyperbolic metric of a domain $G$ by $\Hd_G$. We adopt the convention from \cite{BM2007,KR2013}, so the hyperbolic metric of a domain $G$ is the unique complete conformal metric on $G$ with constant negative curvature $-1$. In particular, the density of the hyperbolic metric of the unit disk $\UD$ is
\begin{equation*}
 \Hd_{\D}(z)=\frac{2}{1-|z|^2}.
\end{equation*}

\begin{theorem} \label{thm:main1} Let $\Delta$ be a hyperbolic petal of a non-trivial one-parameter semigroup $(\phi_t)$ in the unit disk, and let ${z_0\in\Delta}$.
Then the petal $\Delta$ is conformal at its $\alpha$-point $\sigma$ if and only if
\begin{equation}\label{EQ_int-mainthrm1-comvergence}
 \int \limits_{-\infty}^0 \log\, \left(\frac{\Hd_{\Delta}(z_0)}{\Hd_{\D}(\phi_t(z_0)) \, |\phi_t'(z_0)|} \right)\,\, \di t ~<~+\infty.
\end{equation}
In this case the integral in~\eqref{EQ_int-mainthrm1-comvergence} converges for
every~$z_0\in \Delta$,  and in fact locally uniformly in $\Delta$.
\end{theorem}

We note that the second named author and Zarvalis showed recently \cite[Theorem~1.3]{MK-KZ2022} that as~$t\to-\infty$,
\begin{equation} \label{EQ:MariaKonstantinos}
  \Hd_{\UD}(\phi_t(z)) \, |\phi'_t(z)|\,\nearrow\, \Hd_{\Delta}(z)\quad\text{locally uniformly in~$\Delta$.}
\end{equation}
Hence, Theorem \ref{thm:main1} relates the rate of convergence in~\eqref{EQ:MariaKonstantinos} with the conformality of~$\Delta$ at~$\sigma$.

\begin{remark}
It is easy to construct examples of hyperbolic petals~$\Delta$ such that $\partial \Delta$ coincides in a neighbourhood of its $\alpha$-point $\sigma$ with~$\partial\UD$, see e.g. \cite[Examples~7.4 and~7.8]{Bracci_et_al2019}; clearly, in such a case $\Delta$ is conformal at~$\sigma$.  An example of a \textit{non-conformal} hyperbolic petal based on certain subtle properties of the hyperbolic distance was given in \cite[Sect.\,8]{Bracci_et_al2019}. In Remark \ref{rem:PetalExamples} we will describe a simple device which allows a painless construction of numerous examples of conformal as well as non-conformal hyperbolic petals.
\end{remark}

The proof of Theorem \ref{thm:main1} is long and is therefore divided into several steps. We shall require various tools from the general theory of one-parameter semigroups of the unit disk, in particular the Berkson\,--\,Porta theory, holomorphic models and pre-models and basic properties of petals. These tools are collected and explained in a preliminary Section \ref{SEC:Preliminaries}. This section also thoroughly introduces the angular derivative problem. 
 In Section \ref{SEC:Auxiliary} we state and prove two technical, but crucial auxiliary results: an integral criterion for conformality of domains which are starlike at infinity and a lemma on convergence of conformal mappings on the boundary. The proof of Theorem~\ref{thm:main1} is given in Section~\ref{SEC:Result} and is divided into several steps.  In Subsection~\ref{SUB_reformulation} we  state a conformality criterion for the Koenigs domain $\Omega$ of the semigroup (Theorem~\ref{thm:main2}), and show how it implies Theorem~\ref{thm:main1}. In Subsection~\ref{SS_sufficiency} we prove the \hbox{if-part} of Theorem~\ref{thm:main2} by establishing in Theorem~\ref{thm:main2suff} a pointwise lower bound, given in euclidean terms, for the quotient of the hyperbolic densities of a domain $\Omega$ which is starlike at infinity and a maximal strip contained
in $\Omega$. The proof of this lower bound uses a mixture of tools from geometric function theory such as monotonicity of hyperbolic densities, Green's function, harmonic measure and kernel convergence.
The \hbox{only-if} part of Theorem~\ref{thm:main2} is proved in Subsection~\ref{SS_necessity} by comparing the Koenigs domain $\Omega$ of $(\phi_t)$ with carefully chosen slit domains and using potential-theoretic tools. The proof of Theorem~\ref{thm:main2} is finished in Subsection \ref{SUB:compl-of-the-proof} where we show that the previously obtained pointwise estimates in fact hold locally uniformly. In Section~\ref{SEC:parabolic_petals} we state and prove a conformality criterion for the case of a parabolic petal, see Theorem \ref{thm:parabolic}.
In the concluding Section \ref{SEC:Remarks} we discuss how the results of this paper are related to several other recent results, in particular the conformality conditions obtained by Betsakos and Karamanlis \cite{BK2022}. In addition, we indicate some potential alternative approaches to the conformality problem for petals, and raise several questions that remain open.

\section{Preliminaries}\label{SEC:Preliminaries}

\subsection{One-parameter semigroups in the unit disk.}
Here we briefly recall the main definitions and basic facts concerning one-parameter semigroups of holomorphic functions. For more details and proofs of the statements cited in this section we refer interested readers to the monographs~\protect{\cite{BCD-Book,SE2010,Shobook01}, \cite[Chapter~4]{Ab89} and to~\cite[Chapter 4]{ERS_2019}.

For a domain ${D\subset\C}$ and a set ${E\subset\C}$ we denote by $\Hol(D,E)$ the set of all holomorphic functions in~$D$ with values in~$E$. As usual, we endow $\Hol(D,E)$ with the topology of locally uniform convergence. Then $\Hol(D,D)$ becomes a topological semigroup w.r.t. the composition operation ${(\phi,\psi)\mapsto \phi\circ\psi}$.

\begin{definition}
A \textsl{one-parameter semigroup} in the unit disk~$\UD$ is a continuous semigroup homomorphism ${[0,+\infty)\ni t\mapsto \phi_t\in\Hol(\UD,\UD)}$ from the semigroup ${\big([0,+\infty),+\big)}$ with the Euclidean topology to the semigroup $\Hol(\UD,\UD)$.
\end{definition}

Equivalently, a family $(\phi_t)_{t\ge0}\subset\Hol(\UD,\UD)$ is a one-parameter semigroup if and only if it satisfies the following three conditions: (i)~$\phi_0=\id_\UD$; (ii)~$\phi_{s}\circ \phi_t=\phi_{s+t}$ for any $s,t\ge0$; (iii)~$\phi_t\to\id_{\UD}$ in $\Hol(\UD,\UD)$ as $t\to0^+$. Thanks to Montel's normality criterion, see e.g. \cite[\S\,II.7, Theorem~1]{Goluzin}, the continuity condition~(iii) is equivalent to the pointwise convergence: $\phi_t(z)\to z$ as~$t\to0^+$ for each $z\in \UD$. At the same time, in the presence of~(i) and~(ii), condition~(iii) is equivalent to a much stronger property: the map $(z,t)\mapsto \phi_t(z)$ is jointly real-analytic in ${\UD\times[0,+\infty)}$. Moreover, every one-parameter semigroup in~$\UD$ represents the semiflow of a (uniquely defined) holomorphic vector field $G:\UD\to\C$, known as the \textsl{infinitesimal generator} of~$(\phi_t)$. This means that for each fixed ${z\in\UD}$, the function $t\mapsto\phi_t(z)$ is the unique solution to the initial value problem
\begin{equation}\label{EQ_ODE}
 \frac{\di}{\di t}\phi_t(z)=G\big(\phi_t(z)\big),\quad t\ge0;\quad \phi_0(z)=z.
\end{equation}
Using conditions~(i) and (ii) one can easily deduce from~\eqref{EQ_ODE} the following PDE:
\begin{equation}\label{EQ_PDE}
 \frac{\partial\phi_t}{\partial t}=G(z)\phi'_t(z),\quad t\ge0,~z\in\UD.
\end{equation}

\begin{remark}\label{RM_univalent}
It follows from the standard uniqueness results for solutions of ODEs that every element of a one-parameter semigroup is an injective map.
\end{remark}

\begin{remark}
The definition of a one-parameter semigroup can be literally extended to an arbitrary domain ${D\subset\C}$. However, this yields an interesting class of objects only if $D$ is conformally equivalent to~$\UD$ or to ${\UD^*:=\UD\setminus\{0\}}$, with the latter case easily reduced to the former one; see e.g. \cite[\S8.4]{BCD-Book}. Clearly, given that $D$ admits a conformal mapping~$f$ onto~$\UD$, a family ${(\phi_t)_{t\ge0}\subset\Hol(D,D)}$ is a one-parameter semigroup in~$D$ if and only if the mappings $f\circ\phi_t\circ f^{-1}$ form a one-parameter semigroup in~$\UD$.
\end{remark}

Another way to modify the definition of a one-parameter semigroup is to allow negative values of the parameter~$t$. In such a case, we have $\phi_t\circ\phi_{-t}=\id_{\UD}$ for any~$t\in\Real$ and hence we end up with a \textsl{one-parameter group of automorphisms}~$(\phi_t)_{t\in\Real}$.

The classical representation formula due to Berkson and Porta~\cite{BP78} characterizes infinitesimal generators in~$\UD$ as functions of the form
\begin{equation}\label{EQ_BP}
G(z)=(\tau-z)(1-\overline\tau z)p(z),\quad z\in\UD,
\end{equation}
where $p$ is a holomorphic function in~$\UD$ with $\Re p\ge0$ and $\tau$ is a point in the closure of~$\UD$. The function $p$ is uniquely determined by~$G$. The same concerns~$\tau$ unless $G\equiv0$.
\medskip

In order to exclude from consideration certain degenerate cases,  we accept the following\\[.5ex]
{\bf Assumption:} the one-parameter semigroups~$(\phi_t)$ we consider in this paper do not extend to one-parameter groups of automorphisms.
\medskip

This assumption is equivalent to requiring that there exists no  ${t>0}$ such that $\phi_t$ is an automorphism.

The distinguished point~$\tau$ in the Berkson\,--\,Porta representation formula~\eqref{EQ_BP} has a very clear dynamic meaning: ${\phi_t(z)\to\tau}$ locally uniformly in~$\UD$ as ${t\to+\infty}$. If $\tau\in\UD$, then
\begin{equation}\label{EQ_at-the-DW-point}
 \phi_t(\tau)=\tau\quad \phi'_t(\tau)=e^{\lambda t},~{}~\lambda:=G'(\tau),\quad\text{for all~$~t\ge0$},
\end{equation}
with $\Re\lambda<0$.

If $\tau\in\UC$, then \eqref{EQ_at-the-DW-point} holds in the sense of angular limits. Here and in what follows, given a holomorphic function ${f:\UD\to\Complex}$ and a point ${\zeta\in\UC}$, by $f(\zeta)$ we denote the angular limit $\anglim_{z\to\zeta}f(z)~\in~\ComplexE:=\Complex\cup\{\infty\}$. Similarly, if $f(\zeta)$ does exist finitely, then by $f^\prime(\zeta)$ we denote the angular derivative
\begin{equation*}
f^\prime(\zeta):=\anglim_{z\to\zeta}\frac{f(z)-f(\zeta)}{z-\zeta}.
\end{equation*}

\begin{remark}\label{RM_contact}
One special important case, in which the existence of the angular derivative is guaranteed, is when $f\in\Hol(\UD,\UD)$ and $f(\zeta)$ exists and belongs to~$\UC$, see e.g. \cite[Proposition~4.13 on p.\,82]{Pommerenke:BB}. In this case, $\zeta$ is called a \textsl{contact point} for the self-map~$f$; the angular derivative $f'(\zeta)$ at a contact point does not vanish, but it can be infinite.
\end{remark}
A \textsl{boundary fixed point} of $f\in\Hol(\UD,\UD)$ is a contact point~$\zeta$ such that $f(\zeta)=\zeta$. A boundary fixed point (or more generally, a contact point)~$\zeta$ is said to be \textsl{regular} if~${f'(\zeta)\neq\infty}$. Boundary fixed points which are not regular are also called \textsl{super-repulsive} (or \textsl{super-repelling}) fixed points of~$f$. The angular derivative $f'(\zeta)$ at a boundary regular fixed point~$\zeta$  is a positive real number and further two subcases are distinguished: the boundary fixed point~$\zeta$ is \textsl{repulsive} (or \textsl{repelling}) if ${f'(\zeta)>1}$, while for $f'(\zeta)\in(0,1]$, it is called \textsl{attracting}.

\begin{remark}\label{RM_existence-anglim}
It is worth mentioning that for elements of one-parameter semigroups the angular limit $\phi_t(\zeta)$ exists at \emph{every} point~${\zeta\in\UC}$, see \cite{CDP2004,AngLim}. Moreover, the orbit $t\mapsto\phi_t(\zeta)$ is continuous for each~$\zeta\in\UC$. At the same time, the extensions of the holomorphic maps $\phi_t(\cdot)$ to~$\UC$ by angular limits are not necessarily continuous on~$\UC$.
\end{remark}

According to the Denjoy\,--\,Wolff Theorem, a self-map $f\in\Hol(\UD,\UD)\setminus\{\id_\UD\}$ either has an attracting fixed point~${\tau\in\UC}$ and no fixed points in~$\UD$,  or it has a fixed point ${\tau\in\UD}$ and no attracting fixed points on~$\UC$. In both cases, $\tau$ is unique and it is called the \textsl{Denjoy\,--\,Wolff point}  (or \textsl{DW-point} for short) of the self-map~$f$.

From \eqref{EQ_at-the-DW-point} it is clear that $\tau$ in the Berkson\,--\,Porta representation formula~\eqref{EQ_BP} is the DW-point for each $\phi_t$ with~${t>0}$. If~$\tau\in\UD$, then  $(\phi_t)$ is said to be \textsl{elliptic}. If ${\tau\in\UC}$, then $\lambda\le0$, and depending on whether $\lambda<0$ or $\lambda=0$, the one-parameter semigroup $(\phi_t)$ is said to be \textsl{hyperbolic} or \textsl{parabolic}, respectively. By the continuous version of the Denjoy\,--\,Wolff Theorem, $\phi_t(z)\to\tau$ locally uniformly in~$\UD$ as ${t\to+\infty}$.

Similarly to the DW-point, repulsive (and super-repulsive) fixed points are common for all elements of a one-parameter semigroup. More precisely, ${\sigma\in\UC}$ is a repulsive (or super-repulsive) fixed point of~$\phi_t$ for some ${t>0}$ if and only if it is a repulsive (resp., super-repulsive) fixed point of~$\phi_t$ for all~${t>0}$; see e.g.~\cite{CDP2004}.

Fixed points of a one-parameter semigroup can be characterized in terms of the infinitesimal generator. It is known~\cite{CD_RACSAM} that ${\sigma\in\UC}$ is a boundary regular fixed point of~$(\phi_t)$ if and only if $G(\sigma)=0$ and $\lambda:=G'(\sigma)$ exists finitely; see also \cite[Sect.\,12.2]{BCD-Book}. In such a case,\ $\lambda\in\Real$ and $\phi_t'(\sigma)=e^{\lambda t}$ for all~${t\ge0}$.
Clearly, if $\lambda>0$, then $\sigma$ is a repulsive fixed point; otherwise, i.e. if $\lambda\le0$, then $\sigma$ is the DW-point of~$(\phi_t)$.

The following remark contains a useful construction indicating that every elliptic one-parameter semigroup, which is not a group, is correlated with a unique non-elliptic one parameter-semigroup.

\begin{remark}\label{RM_lifting}
Suppose that $(\phi_t)$ is an elliptic semigroup with the DW-point ${\tau\in\UD}$. Then $(\phi_t)$  can be regarded as a one-parameter semigroup in~${\UD\setminus\{\tau\}}$. Consider the covering map ${T\circ C:\UD\to\UD\setminus\{\tau\}}$, where $C(z):=\exp(-\tfrac{1+z}{1-z})$, $T(w):=(w+\tau)/(1+\overline\tau w)$. It is known, see \cite[Sect.\,2]{AngLim}, that there is a (unique) one-parameter semigroup $(\varphi_t)$ which is a lifting of~$(\phi_t)$ w.r.t. $C\circ T$, i.e. such that $\phi_t\circ T\circ C=T\circ C\circ\varphi_t$ for all~${t\ge0}$. Further details and application of the above construction follow in the proof of Theorem \ref{thm:main1}.
\end{remark}

\subsection{Holomorphic models and Koenigs function}\label{SUB:Koenigs}
It is known, see e.g. \cite[Sect.\,9.2]{BCD-Book} or~\cite{BCM2018}, that any one-parameter semigroup admits a \textsl{holomorphic model} $(\Omega_0,h,L_t)$. This means that $\Omega_0\subset\C$ is a simply-connected domain, referred to as the \textsl{base space}, ${h:\UD\to\Omega_0}$ is a injective holomorphic map, and $(L_t)$ is a one-parameter group of holomorphic automorphisms of~$\Omega_0$ with the following two properties:
\begin{align}\label{EQ_general_Abel}
 h\circ\phi_t~=~L_t\circ h, \qquad &\text{for all $t\ge0$};\\
 \bigcup_{t\le0}\,L_t(\Omega)~=~\Omega_0, \qquad &\text{where $\Omega:=h(\UD)$}.\label{EQ_absorbing}
\end{align}
Up to a naturally defined isomorphism, a holomorphic model for a given one-parameter semigroup is unique.

The theory which shall be presented in the current and the upcoming sections can also be generalized (with appropriate modifications) to the case of an elliptic one-parameter semigroup, which is not an elliptic group. Taking Remark \ref{RM_lifting} into consideration, from this point onward, we can safely consider only non-elliptic one-parameter semigroups (i.e. hyperbolic or parabolic).

Non-elliptic one-parameter semigroups admit holomorphic models for which ${L_t(z):=z+t}$, ${t\in\Real}$, and $\Omega_0$ is the whole $\C$ or a half-plane or a strip with $\partial\Omega_0$ composed of one or two lines parallel to~$\Real$. For such holomorphic models, equation~\eqref{EQ_general_Abel} becomes Abel's functional equation
\begin{equation}\label{EQ_Abel}
  h\big(\phi_t(z)\big)=h(z)+t\quad\text{for all~$~z\in\UD~$ and all~$~t\ge0~$}.
\end{equation}
The function $h$ is called the \textsl{Koenigs function} of~$(\phi_t)$ and it is unique up to an additive constant. The set $\Omega:=h(\UD)\subset\Omega_0$ is called the \textsl{Koenigs} (or sometimes, \textsl{planar}) \textsl{domain} of~$(\phi_t)$. Abel's equation implies an important property of~$\Omega$: for every of its points~$w$, the Koenigs domain contains the ray $\{w+t\colon t\ge0\}$. Such domains are said to be \textsl{starlike at infinity}.

Many dynamical properties of~$(\phi_t)$ are encoded in the geometry of the corresponding Koenigs domain~$\Omega$.
Moreover, any starlike-at-infinity domain $\Omega$ different from the whole plane is the Koenigs domain of a non-elliptic one-parameter semigroup. This is often used to construct examples of one-parameter semigroups with given behaviour, see e.g.~Remark \ref{rem:PetalExamples} below.

\begin{definition} We denote by $\St$ the ``standard'' horizontal strip ${\{z:|\Im z|<\tfrac\pi2\}}$, and more generally we denote ${\St(a,b):=\{z:a<\Im z<b\}}$ for $a,b \in \R$ with $a<b$.
Let $\Omega$ be the Koenigs domain of a non-elliptic one-parameter semigroup $(\phi_t)$.
A strip $\St(a,b)$ contained in~$\Omega$ is said to be a \textsl{maximal strip} for~$(\phi_t)$ if $\St(a,b)\subset\St(a',b')\subset\Omega$ holds only for ${(a',b')=(a,b)}$.
\end{definition}
It is easy to see that the maximal strips defined above are connected components of the interior of ${\bigcap_{t\ge0}\Omega+t}$.

\begin{remark}\label{RM_BRFP-max-strips}
It is known~\cite{Analytic-flows} that there exists a one-to-one correspondence between the repulsive fixed points~${\sigma\in\UC}$ of $(\phi_t)$ and the maximal strips in the Koenigs domain of~$(\phi_t)$. If $S$ is a maximal strip for~$(\phi_t)$ and ${w\in S}$, then $h^{-1}(w+t)$ tends, as ${t\to-\infty}$, to the corresponding repulsive fixed point~$\sigma$.  Moreover, the width $\nu(S)$ of the maximal strip~$S$ is related to the angular derivative at~$\sigma$: namely, $\nu(S)G'(\sigma)=\pi$.
\end{remark}

\subsection{Backward orbits, invariant petals, and pre-models}\label{SUB:petals}
In this section we follow the terminology from~\cite{Bracci_et_al2019}. For the proofs of statements quoted below we refer the reader to the same source.
Let us denote by $\hdist_D$ the hyperbolic distance in a hyperbolic domain $D$.

\begin{definition}(\cite[Definition~3.1]{Bracci_et_al2019})
A continuous curve $\gamma:[0,+\infty)$ is called a \textsl{backward orbit} of a one-parameter semigroup~$(\phi_t)$ if for any ${t>0}$ and any $s\in(0,t)$, we have $\phi_s(\gamma(t))=\gamma(t-s)$. A backward orbit~$\gamma$ is said to be \textsl{regular} if $~{\limsup_{t\to+\infty}\hdist_\UD\big(\gamma(t),\gamma(t+1)\big)<+\infty}$.
\end{definition}

\begin{remark}\label{RM_backward-orbit}
Let $(\phi_t)$ be a non-elliptic one-parameter semigroup in~$\UD$.
Fix $z\in\UD$. It is easy to see that the following three conditions are equivalent:
\begin{itemize}
\item[(i)] there exists a backward orbit $\gamma$ with $\gamma(0)=z$;
\item[(ii)] $z\in\mathcal W:=\bigcap_{t\ge0}\phi_t(\UD)$;
\item[(iii)] the line $\{h(z)+t\colon t\in\Real\}$, where $h$ is the Koenigs function of~$(\phi_t)$, is contained in the Koenigs domain~$\Omega$ of~$(\phi_t)$.
\end{itemize}
If the above conditions are satisfied, then the backward orbit $\gamma$ in~(i) is unique and it is given by $\gamma(t):=\phi_t^{-1}(z)=h^{-1}\big(h(z)-t\big)$ for all~$t\ge0$.
Moreover, this backward orbit $\gamma$ is regular if and only if $z \in \inter{\mathcal W}$.
\end{remark}

\begin{remark}\label{RM_for-negative-ts}

The \textsl{negative iterates} $\phi_{-t}:=\phi_t^{-1}$, ${t>0}$, are well-defined and holomorphic in $\inter{\mathcal W}$.
Thus, for $z\in \inter{\mathcal W}$, the differential equations~\eqref{EQ_ODE} and~\eqref{EQ_PDE} are valid for all $t<0$.
\end{remark}

\begin{definition}
The set $\mathcal W$ in Remark~\ref{RM_backward-orbit} is called the \textsl{backward invariant set} of~$(\phi_t)$. Each non-empty connected component of~$\inter{\mathcal W}$ is called a \textsl{petal}.
\end{definition}
Every petal $\Delta$ is a simply connected domain and $\big(\phi_t|_\Delta\big)_{t\in\Real}$ is a one-parameter group of automorphisms of~$\Delta$.
The boundary of~$\Delta$ contains the DW-point~$\tau$ of~$(\phi_t)$.
All regular backward orbits that lie in a petal $\Delta$ converge to the same boundary fixed point of $(\phi_t)$ which lies on the boundary of the petal. We call this unique limit point \textit{the $\alpha$-point of the petal $\Delta$}. The following dichotomy holds:
\begin{itemize}
\item[(P)] either the $\alpha$-point of the petal $\Delta$ coincides with~$\tau$ and it is the only fixed point of $(\phi_t)$ contained in~$\partial\Delta$;
\item[(H)] or $\partial\Delta$ contains exactly two fixed points of the semigroup: the DW-point $\tau$ and a repulsive fixed point~$\sigma$, which is the $\alpha$-point of the petal~$\Delta$.
\end{itemize}
The case~(P) arises only if the one-parameter semigroup is parabolic. In this case, the image $h(\Delta)$ of the petal~$\Delta$ w.r.t. the Koenigs function~$h$ is a half-plane bounded by a line parallel to~$\Real$; it is maximal in the sense that there exist no half-plane $H\neq h(\Delta)$ such that ${h(\Delta)\subset H\subset \Omega}$.

In case~(H), the petal $\Delta$  is said to be \textsl{hyperbolic} and $h(\Delta)$ coincides with the maximal strip corresponding to the repulsive fixed point~$\sigma$ in the sense of Remark~\ref{RM_BRFP-max-strips}. Moreover, there is a one-to-one correspondence\,\footnote{In the case of non-elliptic one-parameter semigroups, this one-to-one correspondence was discovered by Contreras and D\'\i{}az-Madrigal~\cite{Analytic-flows}. Similar results for elliptic and hyperbolic one-parameter semigroups were independently established by Elin, Shoikhet and Zalcman in~\cite{ShoiElinZalc}. A proof covering all the cases and more details can be found in \cite[Sect.\,4]{Bracci_et_al2019}.}\, between the repulsive fixed points and the hyperbolic petals, as the pre-image $h^{-1}(S)$ of any maximal strip~$S$ is a hyperbolic petal.  In what follows, the hyperbolic petal corresponding to a given repulsive fixed point~$\sigma$ will be denoted by~$\Delta(\sigma)$.

The Koenigs function can be regarded as a global change of variables reducing the dynamics of~$(\phi_t)$ to the canonical form $w\mapsto w+t$. When studying dynamics of the one-parameter semigroup in a petal~$\Delta$, instead of the holomorphic model it is more convenient to work with the so-called pre-model. This notion has been introduced  for discrete iteration in~\cite{Poggi-Corradini}. The definition below is a slight modification of that from \cite[Definition~3.8]{Bracci_et_al2019} combined with \cite[Remark~3.9]{Bracci_et_al2019}.
We denote by $\UH$ the right half-plane $\{z \in \C \, : \, \Re z>0\}$.

\begin{definition}\label{DF_premodel}
Let $\sigma\in\UC$ be a repulsive fixed point of a one-parameter semigroup $(\phi_t)$ with associated infinitesimal generator~$G$. The triple $(\UH,\psi,Q_t)$ is called a pre-model for $(\phi_t)$ at~$\sigma$ if the following conditions are met:
\begin{itemize}
\item[(i)] for each $t\ge0$, $Q_t$ is the automorphism of~$\UH$ given by $Q_t(z):=e^{\lambda t}z$, where $\lambda:=G'(\sigma)$;
\item[(ii)] the map $\psi:\UH\to\UD$ is holomorphic and injective, $\anglim_{w\to0}\psi(w)=\sigma$, and $\psi$ is isogonal at~$0$, i.e.
    \begin{equation}\label{EQ_isogonality}
      \anglim_{w\to0}\Arg\frac{1-\overline{\sigma}\psi(w)}w=0;
    \end{equation}
\item[(iii)] $\psi\circ Q_t=\phi_t\circ\psi$ for all~$t\ge0$.
\end{itemize}
\end{definition}

\begin{remark}\label{RM_pre-model}
It is known \cite[Theorem~3.10]{Bracci_et_al2019} that every one-parameter semigroup, at each repulsive fixed point~$\sigma$, admits a pre-model unique up to the transformation $\psi(w) \mapsto \psi(cw)$, where $c$ is an arbitrary positive constant. Moreover, $\psi(\UH)$ is the hyperbolic petal $\Delta(\sigma)$ with $\alpha$-point~$\sigma$.
The map~$\psi$ can be expressed via the Koenigs function~$h$ of~$(\phi_t)$. Namely, if the strip~$h(\Delta(\sigma))=\St(a,b)$, then the map $\psi$ in the pre-model for~$(\phi_t)$ at~$\sigma$ is given by
\begin{equation*}
\psi(w):=h^{-1}\big(\tfrac{b-a}{2\pi}\log w+\tfrac{b+a}{2}i+s\big),\quad w\in\UH,
\end{equation*}
where $s$ is an arbitrary real constant.
\end{remark}

\begin{remark}\label{RM_backward-orbits-non-tangent}
One important consequence of the facts mentioned in Remark~\ref{RM_pre-model} is that every backward orbit~$\gamma$ starting from a point~$z$ in a hyperbolic petal $\Delta(\sigma)$ converges to~$\sigma$ non-tangentially and with a definite slope, i.e. there exists the limit
\begin{equation*}
 \theta(z):=\lim_{t\to+\infty}\Arg\big(1-\overline\sigma\gamma(t)\big),
\end{equation*}
with $\theta(z)\in(-\pi/2,\pi/2)$.
\end{remark}

\subsection{Conformality of a domain at a boundary point}\label{SUB:AngularDer}
The geometry of Koenigs domains of non-elliptic one-parameter semigroups is strongly affiliated to the notion of conformality in the ``strip normalization'' studied in detail, e.g. in~\cite{RodinWarsch}.
\begin{definition}\label{DEF:O-Conformality}
Let $S:=\St(a,b)$ be a maximal strip contained in a domain~$\Omega$ and let $g$ be a conformal mapping of $\Omega$ onto $\St(a,b)$ such that
\begin{equation}\label{EQ_g-normalization}
\Re g(t+iy_0)\to -\infty \quad \text{as~$~t\to-\infty~$}
\end{equation}
for some and hence all~$y_0\in(a,b)$.
The domain~$\Omega$ is said to  \textsl{have an angular derivative} or  to be \textsl{conformal at $-\infty$ w.r.t.~$S$} if for any $\varepsilon\in(0,b-a)$ there exists the finite real limit
\begin{equation}\label{EQ_finite_ang-deriv-strip}
\hbox{~}\qquad\qquad\lim_{D(\varepsilon)\ni z\to\infty} g(z)-z,\quad D(\varepsilon):=\{z\in S\colon \Re z\le 0,~ \dist(z,\partial S)\ge\varepsilon/2\}.
\end{equation}
\end{definition}
\begin{remark}
Clearly, the map $g$  above is not uniquely defined. However, it is easy to see that if condition~\eqref{EQ_finite_ang-deriv-strip} holds for one conformal map $g$ of $\Omega$ onto $\St(a,b)$ satisfying~\eqref{EQ_g-normalization}, then~\eqref{EQ_finite_ang-deriv-strip} holds for all such mappings~$g$.
\end{remark}

In a different geometric setting, it is natural to consider a similar and closely related notion of conformality w.r.t. the unit disk~$\UD$. In this case, we restrict ourselves to subdomains of~$\UD$.

\begin{definition}\label{DEF:ConformalityDiskSubset}
A simply connected domain $U\subset\UD$ is said to be \textsl{conformal at a point} $\sigma\in\partial U \cap\partial\UD$ w.r.t.~$\UD$ if there exists a conformal mapping~$\varphi$ of $\UD$ onto~$U$ such that $\varphi(1)=\sigma$ in the sense of angular limits and the angular derivative $\varphi'(1)$
is finite.
\end{definition}

Note that the condition $\varphi(1)=\sigma$ in the above definition means that ${\zeta=1}$ is a contact point of~$\varphi$; hence, the angular derivative $\varphi'(1)$ exists and does not vanish, but in general, can be infinite; see Remark~\ref{RM_contact}.
To simplify the terminology in the case when $U$ is a petal, we make the following definition.

\begin{definition}\label{DF_petal-conformality}
Let $\Delta(\sigma)$ be a hyperbolic petal of a one-parameter semigroup~$(\phi_t)$ with $\alpha$-point~$\sigma$. We say that $\Delta(\sigma)$ is \textsl{conformal}, if $\Delta(\sigma)$ is conformal at~$\sigma$ w.r.t.~$\UD$.
\end{definition}

\begin{remark}\label{RM_not-depend-on-normalization}
Note that, in general, there can exist two conformal mappings $\varphi_k$, $k=1,2,$ of~$\UD$ onto the same domain~$U\subset\UD$ such that $\varphi_k(1)=\sigma$, $k=1,2$, for some~$\sigma\in\UC$, but with ${\varphi_1'(1)=\infty}$ while $\varphi_2'(1)$ is finite. This phenomenon may happen if the geometric point~$\sigma$ corresponds to at least two different accessible boundary points of~$U$. However, this is never the case for petals, see \cite[Proposition~4.15]{Bracci_et_al2019}. Therefore, in order to determine whether a petal $\Delta(\sigma)$ is conformal it is sufficient to construct \emph{just one} conformal map $\varphi$ of~$\UD$ onto~$\Delta(\sigma)$ with $\varphi(1)=\sigma$ in the sense of angular limits and check whether the angular derivative $\varphi'(1)$ is finite. According to Remark~\ref{RM_pre-model}, every one-parameter semigroup admits a pre-model $(\UH,\psi,Q_t)$ at each repulsive fixed point~$\sigma$. A conformal map of~$\UD$ onto~$\Delta(\sigma)$ taking $1$ to~$\sigma$ is given by $\varphi(z):=\psi\big(\tfrac{1-z}{1+z}\big)$. Therefore, a hyperbolic petal $\Delta(\sigma)$ is conformal if and only if the pre-model $(\UH,\psi,Q_t)$ is \textsl{regular} in the sense that the angular derivative $\psi'(0)$ is finite. Note that this condition is stronger than the isogonality condition~\eqref{EQ_isogonality}, which is also sometimes called \textsl{semi-conformality}; see e.g. \cite[Sect.\,4.3]{Pommerenke:BB}. Condition \eqref{EQ_isogonality} is satisfied in our context by the very definition of a pre-model.
\end{remark}

For a non-elliptic one-parameter semigroup~$(\phi_t)$, the two versions of the angular derivative problem introduced above turn out to be equivalent.
\begin{proposition}\label{PR_confprmal-via-conformal}
In the above notation, a hyperbolic petal $\Delta(\sigma)$ is conformal if and only if the Koenigs domain~$\Omega$ is conformal at~$-\infty$ w.r.t. the maximal strip $~\St(\sigma):=h(\Delta(\sigma))$.
\end{proposition}
\begin{proof}
Using conformal automorphisms of~$\UD$, we may assume that ${\sigma=-1}$ and ${\tau=1}$.
Denote $S:=\St(\sigma)$, $~a:=\inf_{z\in S}\Im z$, $~b:=\sup_{z\in S}\Im z$. Let
\begin{equation*}
  q(z):=\frac{e^{L(z)}-1}{e^{L(z)}+1},\quad\text{where~}~L(z):=\pi\,\frac{z-(b+a)i/2}{b-a}.
\end{equation*}
The function $q$ maps $S$ conformally onto~$\UD$ in such a way that $q(z)\to\tau=1$ as ${\Re z\to+\infty}$ and $q(z)\to\sigma=-1$ as ${\Re z\to-\infty}$.

Then $g:=(h\circ q)^{-1}$ is a conformal mapping of $\Omega$ onto~$S$. Moreover, by Remarks~\ref{RM_BRFP-max-strips} and~\ref{RM_backward-orbit}, for any $y_0\in(a,b)$, the curve ${[0,+\infty)\ni t\mapsto h^{-1}(-t+iy_0)}$ is a backward orbit in $\Delta(\sigma)$ and hence it converges to~$\sigma$ as ${t\to+\infty}$. Therefore, $g$ satisfies the normalization~\eqref{EQ_g-normalization}. By the very definition, it follows that $\Omega$ is conformal at~$-\infty$ w.r.t.~$S$ if and only if $g$ satisfies condition~\eqref{EQ_finite_ang-deriv-strip}.

It is elementary to see that~\eqref{EQ_finite_ang-deriv-strip} is in turn equivalent to the existence of finite angular derivative at~${\sigma=-1}$ for the holomorphic self-map $\varphi:\UD\to\UD$ defined by ${\varphi:=q\circ g\circ q^{-1}}$. At the same time we have
\begin{equation*}
\varphi=q\circ(h\circ q)^{-1}\circ q^{-1}=h^{-1}\circ q^{-1}.
\end{equation*}
Hence $\varphi$ maps $\UD$ conformally onto the hyperbolic petal~$\Delta(\sigma)$.
Using again Remark~\ref{RM_BRFP-max-strips}, we see that the radial limit of~$\varphi$ at $\sigma=-1$ equals~$\sigma$. By Lindel\"of's Theorem (see e.g. \cite[Theorem~1.5.7 on p.\,27]{BCD-Book}) the latter means that $\varphi(\sigma)=\sigma$ in the sense of angular limits. According to Remark~\ref{RM_not-depend-on-normalization}, the existence of the finite angular derivative of $z\mapsto\varphi(-z)$ at~${z=1}$ implies that $\Delta(\sigma)=\varphi(\UD)$ is conformal at~$\sigma$ w.r.t.~$\UD$ and hence we obtain the desired result.
\end{proof}

\begin{remark}[Examples of non-conformal hyperbolic petals] \label{rem:PetalExamples}
With some efforts, an example of a non-conformal hyperbolic petal was constructed in~\hbox{\cite[Sect.\,8]{Bracci_et_al2019}}. We briefly indicate how
one can easily obtain many other examples of conformal and non-conformal hyperbolic petals. Lemma~\ref{lem:WarRodin}, which we prove in the next section, allows one to construct starlike-at-infinity domains $\Omega$ containing a maximal strip~$S$ w.r.t. which $\Omega$ is conformal (or non-conformal) at~$-\infty$. Let $h$ be a conformal mapping of~$\UD$ onto~$\Omega$. Then ${\Delta:=h^{-1}(S)}$ is a hyperbolic petal for the non-elliptic semigroup~$(\phi_t)$ given by ${\phi_t:=h^{-1}\circ(h+t)}$, $t\ge0$. By  Proposition~\ref{PR_confprmal-via-conformal}, the petal~$\Delta$ is conformal (respectively, non-conformal). Note that, using the trick described in Remark~\ref{RM_lifting}, this technique can be extended to elliptic semigroups as well.
\end{remark}

\section{Auxiliary results}\label{SEC:Auxiliary}

\subsection{Strong Markov property for the Green's function}\label{SUB_Markov}
Let $D\subsetneq\C$ be a simply connected domain. Let $\Gf_D$ and $\Hm_D$ denote the (positive) Green's function and the harmonic measure for the domain~$D$, respectively; see e.g. \cite[Chapter~4]{Ransford}. In the course of the proofs, we make use of the following remarkable property of the Green's function $\Gf_D$, see \cite[p.\,111]{PortStone}.
\begin{lemma}[Strong Markov Property for the Green's function]
\label{LEM:MarkovGreen}
Let $D_1$ and $D_2$ be two simply connected domains with $D_1\subset D_2\subsetneq\C$.
Then for all $z,w \in D_1$, $z\neq w$,
\begin{equation}\label{EQ_MarkovGreen}
 \Gf_{D_2}(z,w) -\Gf_{D_1}(z,w) = \int_A \Gf_{D_2}(\alpha , z)~\Hm_{D_1} (w, \di \alpha),
\end{equation}
where $A:=D_2\cap \partial D_1$.
\end{lemma}

In the proof of the main results, we occasionally replace the hyperbolic density with the conformal radius; the reader may refer to \cite[\S2.1]{DubininsBook}. For simply connected domains $D\subsetneq \C$ the conformal radius $\confr{z_0}{D}$ of $D$ w.r.t. the point $z_0 \in D$ is just the reciprocal of the hyperbolic density $\Hd_D(z_0)$; namely $\confr{z_0}{D}= 2/\Hd_D(z_0)$.

It is further known that $\Gf_{D}(z,w)+\log|w-z|\to \log\confr{w}{D}$ as ${z\to w\in D}$. Therefore, \eqref{EQ_MarkovGreen} implies
\begin{equation}\label{EQ_MarkovConfRad}
\log\frac{\confr{w}{D_2}}{\confr{w}{D_1}} = \int_A \Gf_{D_2}(\alpha , w)~\Hm_{D_1}(w, \di \alpha),\quad w\in D_1.
\end{equation}

\subsection{An integral criterion for conformality of domains starlike at infinity} \label{sec:conf_int}

Let $\Omega$ be a domain which is starlike at infinity and denote by $S=\St(a,b)=\{x+iy\colon a<y<b\}$ a maximal horizontal strip contained in $\Omega$.

One of the ingredients of the proof of Theorem \ref{thm:main2} is a
characterization of conformality of $\Omega$ at the boundary point $-\infty$ w.r.t.~to $S$  in euclidean terms.
 Such a characterization  can easily be deduced from the work of Rodin \&
Warschawski \cite{RodinWarsch} and Jenkins \& Oikawa \cite{JO1977} as follows.

Fix a point $w_0=iy_0\in S$, $y_0 \in (a,b)$, and denote
\begin{equation*}
\delta_{\Omega,1}(t):=a\,-\,\inf I(t),\quad \delta_{\Omega,2}(t):=\sup I(t)\,-\,b\, ,
\end{equation*}
where $I(t)$ is the connected component of $\{y\colon t+iy\in\Omega\}$ containing~$y_0$. It is clear that ${\delta_{\Omega,j}(t)\ge0}$ for all ${t\in\Real}$, ${j=1,2}$, because ${S\subset\Omega}$. We denote
\begin{equation} \label{EQ_delta_def}
\delta_{\Omega}(t):=\max \left\{\delta_{\Omega,1}(t),\delta_{\Omega,2}(t)\right\} \, .
  \end{equation}
Note that $\delta_{\Omega,1}(t)$, $\delta_{\Omega,2}(t)$ and $\delta_{\Omega}(t)$ do not depend on the choice of the base point $w_0$.
We should also note that there might exist $t\in \mathbb{R}$ such that $\delta_{\Omega}(t)=+\infty$. Due to maximality of the strip $S$ inside $\Omega$, there always exists a $t_0 \in \mathbb{R}$ such that $\delta_{\Omega}(t)<+\infty$ for all $t\leq t_0$. In fact, starlikeness of $\Omega$ at infinity assures that $\delta_{\Omega}$ is monotonically non-decreasing in $(-\infty,t_0]$ and that $\delta_{\Omega}(t) \to 0$  as $t\to -\infty$.

Since the Koenigs function $h$ is defined modulo an additive constant, by shifting the domain~$\Omega$ along the real axis, in our proofs we may assume without loss of generality that $t_0=1$. The following lemma relates the conformality of the starlike-at-infinity domain $\Omega$ at $-\infty$ w.r.t.~the strip $S$ with the integrability of the euclidean quantity~$\delta_{\Omega}$.

\begin{lemma} \label{lem:WarRodin}
In the above notation, $\Omega$ is conformal at~$-\infty$ w.r.t. the strip~$S$ if and only if
\begin{equation}\label{EQ_integrability-delta}
\int\limits_{-\infty}^0\delta_{\Omega}(t)\,\di t~<~+\infty.
\end{equation}
\end{lemma}
\begin{proof}
Suppose that~\eqref{EQ_integrability-delta} holds. Let $(Q_n)_{n\ge 0}$ be a sequence of pairwise disjoint squares contained in ${(\Omega\setminus S)\cap\{z\colon \Re z\le 0\}}$ and such that for each $Q_n$ one of the sides is contained on~$\partial S$.
From~\eqref{EQ_integrability-delta}, it follows easily that $\sum_{n\ge 0}\mathrm{area}(Q_n)<+\infty$. Hence, by \cite[Theorem 2]{RodinWarsch}$, \Omega$ is conformal at $-\infty$ w.r.t.~$S$.

Now suppose that $\Omega$ is conformal at $-\infty$ w.r.t.~$S$. Since $\Omega+x\subset\Omega$ for any~$x\ge0$, the functions $\delta_j:=\delta_{\Omega,j}$, $j=1,2$, are monotonically non-decreasing.
Therefore, by the implication (i)~$\Rightarrow$~(iii) of \cite[Theorem 2]{RodinWarsch}, there is an increasing sequence ${0=u_0<u_1<\ldots<u_n< \ldots}$ tending to $+\infty$ such that
\begin{equation*}
 \sum_{n=0}^{\infty} (u_{n+1}-u_n)^2<\infty\quad\text{and}\quad\sum_{n=0}^{\infty} \delta_j(-u_{n+1})^2  <+\infty,~\text{~$j=1,2$.}
\end{equation*}

With the help of the Cauchy\,--\,Schwarz\,--\,Bunyakovsky  inequality, it follows that
\begin{equation*}
\int \limits_{-\infty}^0 \delta_j(t) \, \di t ~=~ \sum \limits_{n=0}^{\infty} \int \limits_{-u_{n+1}}^{-u_n} \delta_j(t) \, \di t~\le~\sum \limits_{n=0}^{\infty} (u_{n+1}-u_n) \cdot
\delta_j(-u_{n+1})~<~+\infty
\end{equation*}
   by the monotonicity of $\delta_j$.
\end{proof}

\subsection{A lemma on convergence of Riemann mappings on the boundary}
Let $(D_n)$ be a sequence of simply connected domains in~$\C$ and $\mathcal B$ an open subarc of~$\UC$ with the following properties: $\UD\subset D_n$ and $\mathcal B\subset\partial D_n$ for each $n\in\Natural$. Denote by $\sigma_k$, $k=1,2$, the end-points of the arc~$\mathcal B$.

\begin{lemma}\label{LM_convergence-on-the-boundary-arc}
In the above notation, suppose additionally that each $D_n$ is a  Jordan domain. For each $n\in\Natural$, let $f_n$ denote the conformal map of~$\UD$ onto~$D_n$ normalized by $f_n(0)=0$, $f_n'(0)>0$ and extended by continuity to a homeomorphism between the closures.  If $(D_n)$ converges to~$\UD$ w.r.t.~$0$ in the sense of kernel convergence, then $f_n^{-1}(\sigma_k)\to\sigma_k$, $k=1,2$, as ${n\to+\infty}$.
\end{lemma}
\begin{remark}
Requiring that each of the domains $D_n$ is a Jordan domain is not really essential in the above lemma, but this condition is satisfied in the setting for which we will apply Lemma~\ref{LM_convergence-on-the-boundary-arc}. For more general domains both the statement of the lemma and its proof would become more technical. On the other hand, the assumption that $\UD\subset D_n$ for all $n\in\Natural$ seems to play an important role. Again, in our setting this assumption holds, but without it the conclusion of Lemma \ref{LM_convergence-on-the-boundary-arc} may fail as the  following example demonstrates. The sequence of Jordan domains
\begin{equation*}
 D_n:=\big\{z\in\C:|z|<1-\tfrac1n\big\}\cup\big\{z\in\UD:|\arg z|<\tfrac\pi{2n}\big\}\cup\big\{z\in\UD:\Re z>0,~|z|>1-\tfrac{1}{2n}\big\}\subset\UD
\end{equation*}
converges to $\UD$ w.r.t.~$0$ in the sense of kernel convergence. Although the arc $\mathcal B:=\{z\in\UC:\Re z>0\}$ lies on the boundary of $D_n$ for all $n\in\Natural$ and although the functions $f_n$ defined as above converge locally uniformly in~$\UD$ to the identity map, the pre-images $f_n^{-1}(\mathcal B)$ shrink as ${n\to+\infty}$ to the point $\sigma_0=1$ rather than converge to $\mathcal B$.
\end{remark}
\begin{proof}[\proofof{Lemma~\ref{LM_convergence-on-the-boundary-arc}}]
Denote $g_n:=f_n^{-1}\big|_{\overline\UD}$, where $\overline\UD$ denotes the closure of~$\UD$.  By Cara\-th\'eodory's kernel convergence theorem (\cite[Theorem 1.8]{Pommerenke:BB}), $(f_n)$ and $(g_n)$ converge  locally uniformly in~$\UD$ to the identity mapping. Moreover, the restrictions $g_n\big|_\UD$ can be extended by the Schwarz reflection principle to conformal mappings $g_n^*$ of $\UD_{\mathcal B}:=\{z\in\Complex:|z|\neq1\}\cup\mathcal B$ into~$\C$. Recall that ${g_n(0)=0}$ for all $n\in\Natural$. Moreover, by the locally uniform convergence of $(g_n)$ in~$\UD$, the sequence $|g'_n(0)|$ is bounded. It follows that the extended functions $g^*_n$ form a normal family in~$\UD_{\mathcal B}$ and hence $g^*_n\to\id_{\UD_{\mathcal B}}$ locally uniformly in $\UD_{\mathcal B}$. This fact, however, does not imply on its own the conclusion of the lemma, because $\sigma_k\not\in\UD_{\mathcal B}$. On the other hand, $g_n(w)=g_n^*(w)$ for all ${w\in\mathcal B}$ and all~${n\in\Natural}$ and hence we may conclude that $(g_n)$ converges uniformly on~any closed subarc of  $\mathcal B$ to the identity mapping.

Consider the sequence $h_n(w):=\sigma_0 g_n(w)/g_n(\sigma_0)$, $w\in\overline{\UD}$, where $\sigma_0$ is the midpoint of the arc~$\mathcal B$. Note that $h_n(\UD)\subset\UD$, $h_n(\mathcal B)\subset\UC$, and $h_n(\sigma_0)=\sigma_0$. Therefore, by Loewner's lemma (\cite[Proposition 4.15]{Pommerenke:BB}), $\mathcal B\subset h_n(\mathcal B)$. Since $g_n(\sigma_0) \to \sigma_0$, it is enough to show that $h_n(\sigma_k)\to\sigma_k$ as ${n\to+\infty}$, $k=1,2$. Suppose this is not the case. Then, passing if necessary to a subsequence, we may assume that there exists an open arc $\mathcal C$ on~$\UC$ such that $\mathcal C\not\subset\mathcal B$, and
\begin{equation}\label{EQ_bigger-arc_C}
 \mathcal C\subset h_n(\mathcal B)\quad\text{ for all~$n\in\Natural$.}
\end{equation}
In particular, $h_n^{-1}(\mathcal C)\subset\UC$ for any~${n\in\Natural}$. The functions $h_n^{-1}$ are restrictions of the maps   $z\mapsto {f_n\big(zg_n(\sigma_0)/\sigma_0\big)}$ to~$\overline{\UD}$. Therefore, arguing as above we see that $h_n^{-1}\to \id_{\mathcal C}$ on~$\mathcal C$. Since $\mathcal C\not\subset\mathcal B$, it follows that $h^{-1}_n(\mathcal C)\not\subset\mathcal B$ for $n\in\Natural$ large enough. To complete the proof it remains to notice that the latter conclusion contradicts~\eqref{EQ_bigger-arc_C}.
\end{proof}

\section{Proof of the main result}\label{SEC:Result}
\subsection{Reformulation of the problem}\label{SUB_reformulation}

In this section we reduce Theorem \ref{thm:main1} to showing that if a domain
$\Omega$ is starlike at infinity, then its conformality at~$-\infty$
w.r.t.~a maximal strip $\St(a,b)\subset\Omega$ is equivalent to a certain
condition on how fast $\Hd_{\Omega}$ approaches $\Hd_{\St(a,b)}$
along a horizontal ray ${\{t+iy_0\colon t\le0\}}\subset\St(a,b)$ as $t \to -\infty$. The  precise statement of this result is as follows.

\begin{theorem} \label{thm:main2}
Let $\Omega\subset\C$, $\Omega\neq\C$, be a domain starlike at infinity and
let $S:=\St(a,b)$ be a maximal strip contained in~$\Omega$. Fix a point $w_0 \in S$. Then $\Omega$ is conformal at~$-\infty$ w.r.t.~$S$ if and only if
\begin{equation}\label{EQ_int-mainthrm2-comvergence}
\int\limits_{-\infty}^{0} \log\,\frac{\Hd_{S}(t+w_0)}{\Hd_\Omega(t+w_0)}\,\di t~<~+\infty.
\end{equation}
In this case, the integral (\ref{EQ_int-mainthrm2-comvergence}) converges for
every~$w_0\in S$,  and in fact locally uniformly on $S$.
\end{theorem}

In this subsection, we show that Theorem \ref{thm:main2} implies Theorem~\ref{thm:main1}. This is not difficult for non-elliptic semigroups $(\phi_t)$ in $\D$, but less obvious in the elliptic case.
The proof of Theorem~\ref{thm:main2}
will be given in Subsections~\ref{SS_sufficiency} and~\ref{SS_necessity}.

\begin{proof}[\proofof{Theorem~\ref{thm:main1}}]
First we prove Theorem \ref{thm:main1}  for the case of a non-elliptic semigroup~$(\phi_t)$ in $\D$.
Differentiating Abel's equation~\eqref{EQ_Abel} w.r.t.~$z$, we obtain
\begin{equation}\label{EQ_der-via-Kf}
{h'\big(\phi_t(z)\big)}\phi_t^\prime(z)={h'(z)}
\end{equation}
for all $z\in\UD$, $t\ge0$. Moreover, in view of Remark~\ref{RM_for-negative-ts}, \eqref{EQ_der-via-Kf} holds also for all~${t<0}$ if ${z\in\inter{\mathcal W}}$.

Denote by $\Omega$ the Koenigs domain of~$(\phi_t)$ and by $S=\St(a,b)$ the maximal strip in~$\Omega$ associated to the $\alpha$-point~$\sigma$ of a hyperbolic petal $\Delta=\Delta(\sigma)$, see Remark~\ref{RM_BRFP-max-strips}.
Then $\Omega$ is starlike at infinity and the Koenigs function~$h$ maps $\UD$ conformally onto $\Omega$ and $\Delta(\sigma)$ onto~$S$, see \cite[Theorem~13.5.5]{BCD-Book}. Fix  ${z_0\in\Delta(\sigma)}$. Taking into account that the hyperbolic metric is invariant under conformal mappings and using equality~\eqref{EQ_der-via-Kf} we see that the integrand in~\eqref{EQ_int-mainthrm1-comvergence} can be written as
\begin{equation*}
 \log\, \frac{\Hd_{\Delta}(z_0)}{\Hd_{\D}(\phi_t(z_0)) \, |\phi_t'(z_0)|}=%
 \log\,\frac{|h'(z_0)|\,\Hd_{S}\big(h(z_0)\big)}{|h'(\phi_t(z_0))|\,\Hd_{\Omega}\big(h(\phi_t(z_0))\big) \, |\phi_t'(z_0)|}=%
  \log\,\frac{\Hd_{S}\big(h(z_0)\big)}{\Hd_{\Omega}\big(h(\phi_t(z_0))\big)},
\end{equation*}
which is equal  for any~${t\in\Real}$, according to Abel's equation~\eqref{EQ_Abel} and Remark~\ref{RM_for-negative-ts}, to
\begin{equation*}
 \log\,\frac{\Hd_S\big(h(z_0)\big)}{\Hd_\Omega\big(h(z_0)+t\big)}.
\end{equation*}
 Furthermore, obviously $\Hd_S(h(z_0))=\Hd_S(h(z_0)+t)$ for any~${t\in\Real}$.
 Consequently, for $w_0:=h(z_0)$, the integral in~\eqref{EQ_int-mainthrm1-comvergence} is identical to the integral in~\eqref{EQ_int-mainthrm2-comvergence}. Thus, for non-elliptic semigroups Theorem~\ref{thm:main1} follows from Theorem~\ref{thm:main2} and Proposition~\ref{PR_confprmal-via-conformal}.

Now we show how the elliptic case can be reduced to the non-elliptic case. Consider a semigroup $(\phi_t)$ in $\D$ with DW-point~$\tau\in\UD$. According to Remark~\ref{RM_lifting}, there exists a parabolic semigroup $(\varphi_t)$ in $\D$ such that for all $t\ge0$,
\begin{equation}\label{EQ_phi-varphi}
  \phi_t\circ F=F\circ\varphi_t,
\quad\text{where~}~F:=T\circ C,~\, C(z):=\exp\left(-\frac{1+z}{1-z}\right),~\, T(w):=\frac{w+\tau}{1+\overline\tau w}.
\end{equation}
Note that $F(\zeta_1)=F(\zeta_2)$ if and only if $L^{\circ n}(\zeta_1)=\zeta_2$ for some~$n\in\mathbb Z$, where $L$ is the automorphism of~$\UD$ given by $L(\zeta):=H^{-1}\big(H(\zeta)+2\pi i\big)$, $H(z):=(1+z)/(1-z)$, and $L^{\circ n}$ denotes the $n$-the iterate of $L$. The semigroup $(\varphi_t)$ satisfies for each ${t\ge0}$ the functional equation $\varphi_t\circ L=L\circ\varphi_t$. This is clear from the construction given in~\cite[Sect.\,2]{AngLim}. Therefore, if $\zeta\in\varphi_t(\UD)$ for some $t\ge0$, then $\varphi_t(\UD)$ contains all the points $\zeta'\in\UD$ satisfying ${F(\zeta')=F(\zeta)}$.
With the notation
\begin{equation*}
\mathcal W:=\bigcap_{t\ge0}\phi_t(\UD),\qquad \mathcal U:=\bigcap_{t\ge0}\varphi_t(\UD),
\end{equation*}
it follows that
\begin{equation*}
 F(\mathcal U)=\bigcap_{t\ge0}F(\varphi_t(\UD))=\bigcap_{t\ge0}\phi_t(F(\UD))
 =\bigcap_{t\ge0}\phi_t(\UD\setminus\{\tau\})
 =\mathcal W\setminus\{\tau\}.
\end{equation*}
Note that $\Delta(\sigma)$ is a simply connected domain and $\tau\in \partial \Delta(\sigma)$; see \cite[Proposition~13.4.2]{BCD-Book}. Moreover, by the definition of a petal, $\Delta(\sigma)$ is a connected component of the interior of~$\mathcal W$. Since ${F:\UD\to\UD\setminus\{\tau\}}$ is a covering map, it follows that there exists a connected component $D$ of the interior of~$\mathcal U$ such that $F$ maps $D$ conformally (and in particular, injectively) onto~$\Delta(\sigma)$. By the very definition, $D$ is a petal for~$(\varphi_t)$.

Moreover, by \cite[Proposition~13.4.9]{BCD-Book}, $D$ is a Jordan domain and $\partial\Delta(\sigma)$ is locally connected. It follows, see e.g. \cite[Sect.\,2.2]{Pommerenke:BB},  that $F_*:=F|_D$ extends continuously to~$\partial D$ and that ${F_*(\partial D)=\partial\Delta(\sigma)}$. Recall that $\sigma\in\partial\Delta(\sigma)$.  Therefore, there exists a point $\varsigma\in\partial D$ such that $F_*(\varsigma)=\sigma$. Since $F$ is continuous with $|F|<1$ in~$\UD$ and since ${\sigma\in\UC}$, we have ${\varsigma\in\UC}$. We claim that, $\varsigma\neq1$. Indeed, let $\Gamma$ be a Jordan arc in $D\cup\{1\}$ with one of the end-points at~$1$. Since $F_*$ is continuous in the closure of~$D$, we have  $F(z)\to F_*(1)$ as $\Gamma\ni z\to1$. Taking into account that the only asymptotic value of ${\UH\ni\zeta\mapsto e^{-\zeta}}$ at~$\infty$ is~$0$, it follows that $F_*(1)=\tau\neq\sigma$.

Note that $F$ extends holomorphically to any point of $\UC\setminus\{1\}$. Hence $F(\varsigma)=\sigma$. As a consequence, using Remark~\ref{RM_existence-anglim}, it is easy to see that $\varsigma$ is a repulsive fixed point of~$(\varphi_t)$.

Thus we have constructed a hyperbolic petal~$D$ for $(\varphi_t)$ with $\alpha$-point~$\varsigma$ such that $F$ maps $D$ conformally onto $\Delta(\sigma)$, with $F(\varsigma)=\sigma$.
Since $F$ is holomorphic at~$\varsigma$ and $F'(\varsigma)\neq0$, the petal $D$ is conformal if and only if the petal $\Delta(\sigma)$ is conformal.

It remains to show that condition~\eqref{EQ_int-mainthrm1-comvergence} for $(\phi_t)$ and a point~$z_0\in\Delta(\sigma)$ is equivalent to \eqref{EQ_int-mainthrm1-comvergence} with $(\phi_t)$, $\Delta(\sigma)$, and $z_0$ replaced by $(\varphi_t)$, $D$, and $\zeta_0:=F_*^{-1}(z_0)$, respectively.

First of all since $F$ maps $D$ conformally onto~$\Delta(\sigma)$, with $F(\zeta_0)=z_0$, we have $\Hd_D(\zeta_0)=|F'(\zeta_0)|\Hd_{\Delta(\sigma)}(z_0)$. Furthermore, by~\eqref{EQ_phi-varphi},
\begin{equation*}
 \phi_t(z_0)=F\big(\varphi_t(\zeta_0)\big)\quad\text{~and~}\quad\varphi'_t(\zeta_0)=\phi'_t(z_0) \frac{F'(\zeta_0)}{F'\big(\varphi_t(\zeta_0)\big)}
\end{equation*}
for all $t\in\Real$. Using these relations we obtain
\allowdisplaybreaks
\begin{align} \nonumber
 \frac{\Hd_D(\zeta_0)}{\Hd_{\D}(\varphi_t(\zeta_0)) \, |\varphi_t'(z_0)|}%
 ~&=~\frac{|F'(\zeta_0)|\,\Hd_{\Delta(\sigma)}(z_0)}{\Hd_{\D}\big(\varphi_t(\zeta_0)\big)}\,
 \frac{|F'\big(\varphi_t(\zeta_0)\big)|}{|\phi'_t(z_0)F'(\zeta_0)|}~=~%
  \\ \label{eq:thm2}
  &=\frac{\Hd_{\Delta(\sigma)}(z_0)}{\Hd_{\D}(\phi_t(z_0)) \, |\phi_t'(z_0)|}\,\frac{|F'\big(\varphi_t(\zeta_0)\big)|\,\Hd_{\D}(\phi_t(z_0))}%
 {\Hd_{\D}\big(\varphi_t(\zeta_0)\big)}%
 \\ \nonumber &=\frac{\Hd_{\Delta(\sigma)}(z_0)}{\Hd_{\D}(\phi_t(z_0)) \, |\phi_t'(z_0)|}\,\frac{|F'\big(\varphi_t(\zeta_0)\big)|\,\Hd_{\D}\big(F(\varphi_t(\zeta_0))\big)}%
 {\Hd_{\D}\big(\varphi_t(\zeta_0)\big)}
\end{align}
for all $t\in\Real$. The backward orbit $\gamma(t):=\varphi_{-t}(\zeta_0)$ of $(\varphi_t)$ converges to~$\varsigma$ at an exponential rate, that is,
\begin{equation} \label{eq:convrate}
 \lim \limits_{t \to -\infty} \frac{1}{t} \log \left(1-\overline{\varsigma}\varphi_t(\zeta_0) \right)=\lambda \in (0,+\infty) \, ,
\end{equation}
see~\cite[Proposition 13.4.14]{BCD-Book}, and
non-tangentially, see Remark~\ref{RM_backward-orbits-non-tangent}. Moreover, for any~$w\in\UD$,
 \begin{equation*}
  0 \ge \log \frac{|F'(w)|\,\Hd_{\D}\big(F(w)\big)}{\Hd_{\D}(w)}=%
 \log\frac{|C'(w)|\,\Hd_{\D}\big(C(w)\big)}{\Hd_{\D}(w)}=%
 \log\frac{\mu(w)}{\sinh\mu(w)} \ge -\frac{\mu(w)^2}{6},
 \end{equation*}
 where $\mu(w):=(1-|w|^2)/(|1-w|^2)$. The first equality holds because ${F=T\circ C}$ with ${T\in\Aut(\UD)}$, the second one comes from~\eqref{EQ_phi-varphi} by direct calculation, and the inequality sign can be established by comparing the Maclaurin expansion of ${\sinh(x)/x}$ with that of $e^{\,x^2/6}$. Note that ${1-|\varphi_t(\zeta_0)|^2}<{2\,|1-\overline{\varsigma}\varphi_t(\zeta_0)|}$ for all ${t<0}$ and that $|1-\varphi_t(\zeta_0)| \xrightarrow{t\to -\infty} |1-\varsigma| \neq 0$. Hence, it follows from~\eqref{eq:convrate} that $\int_{-\infty}^0 \mu(\varphi_t(\zeta_0))^2 \di t < +\infty$,
and so we can conclude that the integral
 \begin{equation*}
 \int \limits_{-\infty}^0 \log  \frac{|F'\big(\varphi_t(\zeta_0)\big)|\,\Hd_{\D}\big(F(\varphi_t(\zeta_0))\big)}{\Hd_{\D}\big(\varphi_t(\zeta_0)\big)} \, \di t
 \end{equation*}
 converges.

  In fact, the above integral converges locally uniformly w.r.t.~$\zeta_0 \in D$ because the limit~\eqref{eq:convrate} is attained locally uniformly in~$D$, which in turn follows from the fact that the values of the holomorphic functions
 \begin{equation*}
   D\ni\zeta\,\mapsto\, \frac{1}{t} \log \left(1-\overline{\varsigma}\varphi_t(\zeta) \right),\quad t<-1,
 \end{equation*}
 lie in the strip~$\St$, and hence these functions form a normal family in~$D$. In view of (\ref{eq:thm2}), the proof that Theorem~\ref{thm:main2} implies Theorem~\ref{thm:main1} for elliptic semigroups is therefore reduced to the previous case of non-elliptic semigroups.
\end{proof}

\subsection{Proof of the if-part of Theorem \ref{thm:main2}: condition~\eqref{EQ_int-mainthrm2-comvergence} implies conformality}\label{SS_sufficiency}
As a matter of taste, we prefer to work with the conformal radius instead of the density of the hyperbolic metric. Therefore,
 condition~\eqref{EQ_int-mainthrm2-comvergence} in Theorem~\ref{thm:main2} can be rewritten as
\begin{equation}\label{EQ-via-conformal-radius}
 \int\limits_{-\infty}^{0} \log\,\frac{\confr{t+w_0}{\Omega}}{\confr{t+w_0}{S}}\,\di t~<~+\infty \, , \quad w_0 \in S.
\end{equation}

In order to simplify notation, throughout the current and the following section we restrict ourselves to the case $\Re w_0 =0$.
The following theorem is the key result of this section. It shows that the integrand in (\ref{EQ-via-conformal-radius}) can be estimated from below by  the euclidean quantity $\delta_{\Omega}(t)$ introduced in (\ref{EQ_delta_def}).

\begin{theorem} \label{thm:main2suff}
  Let $\Omega$ be a proper subdomain of $\C$ which is starlike at infinity and such that the standard strip $\St=\{z \in \C \, : \, |\Im z| <\frac{\pi}{2}\}$ is a maximal horizontal strip contained in $\Omega$. Then for any compact set $K \subset (-\frac{\pi}{2},\frac{\pi}{2})$ there are constants $c>0$ and $T \le 0$ such that for any $y \in K$,

\begin{equation} \label{eq:main2suff}
  \log \frac{\mathcal{R}(t+iy,\Omega)}{\mathcal{R}(t+iy,\St)}  \ge c \,  \delta_{\Omega}(t)  \quad \text{ for all~}~t \le T.
  \end{equation}
\end{theorem}

It is clear from our previous considerations that Theorem~\ref{thm:main2suff}
implies the ``if-part'' of Theorem~\ref{thm:main2}. Indeed, if  condition
(\ref{EQ_int-mainthrm2-comvergence}) holds for some $w_0=i y_0 \in S$, $y_0 \in (-\frac{\pi}{2},\frac{\pi}{2})$, then Theorem~\ref{thm:main2suff} and Lemma~\ref{lem:WarRodin} imply that $\Omega$ is conformal at $-\infty$.

The proof of Theorem \ref{thm:main2suff} is quite long and will be broken into several steps.
The idea of the proof is as follows.

\begin{remark}[Idea of proof of Theorem \ref{thm:main2suff}]
Recall that ${\UH=\{z:\Re z>0\}}$ denotes the right half-plane. For $\delta>0$ we consider the enlarged strip
\begin{equation*}
 \St_{\delta}:=\left\{z:-\frac\pi2<\Im z<\frac\pi2+\delta\right\}
\end{equation*}
and the ``half-widened'' standard strip
\begin{equation*}
 \quad \St_{\delta}^*:=\St\cup\big(\St_{\delta}\cap\UH\big).
\end{equation*}

Fix a point $iy_0 \in \Omega$, $y_0 \in (-\frac{\pi}{2},\frac{\pi}{2})$. By definition of $\delta_2(t):=\delta_{\Omega,2}(t)$, see Subsection \ref{sec:conf_int}, and since $\Omega$ is starlike at infinity, it follows at once that
\begin{equation*}
  \St^*_{\delta_2(t)}+t \subset \Omega.
\end{equation*}
 Thus, by domain monotonicity of the conformal radius,
\begin{equation} \label{eq:main2suffidea0}
\log \frac{\mathcal{R}\Big(t+i y_0,\Omega\Big)}{\mathcal{R}\left(t+i y_0,\St\right)} \ge
\log \frac{\mathcal{R}\left(t+i y_0,\St^*_{\delta_2(t)}+t\right)}{\mathcal{R}\left(t+i y_0,\St\right)}
=\log \frac{\mathcal{R}\left(i y_0,\St^*_{\delta_2(t)}\right)}{\mathcal{R}\left(i y_0,\St\right)}, \qquad  t \le 0 \, .
\end{equation}
The crux of the proof is to show that the expression on the r.h.s. of (\ref{eq:main2suffidea0}) is bounded below~by
\begin{equation} \label{eq:main2suffidea}
\log \frac{\mathcal{R}(iy_0,\St_{\delta_2(t)})}{\mathcal{R}(i y_0,\St)},
\end{equation}
at least up to a positive multiplicative constant which depends only on $y_0$ but in a fairly controllable way.
The quantity occurring in (\ref{eq:main2suffidea}) is explicitly computable in terms of $\delta_2(t)$ (and $y_0)$, and as we shall see, in fact comparable to $\delta_2(t)$ locally uniformly w.r.t.~$y_0$. This provides a lower bound also for the l.h.s. in (\ref{eq:main2suffidea0}) in terms of $\delta_2(t)=\delta_{\Omega,2}(t)$ and then, by symmetry, in terms of $\delta_{\Omega}(t)=\max\{\delta_{\Omega,1}(t),\delta_{\Omega,2}(t) \}$, and (\ref{eq:main2suff}) follows.
\end{remark}

\begin{proposition}\label{PR_one-half}
  For $\delta>0$ let $~\Phi_{\delta} : (-\frac{\pi}{2},\frac{\pi}{2}) \to (0,+\infty)$ be defined by
   \begin{equation} \label{EQ_one-half_def}
    \Phi_{\delta}(y):={\log\displaystyle\frac{\confr{iy}{\St^*_\delta}}{\confr{iy}{\St}}}\left( \log\displaystyle\frac{\confr{iy}{\St_{\delta}}}{\confr{iy}{\St}}\right)^{-1}.
  \end{equation}
  Then
  \begin{equation} \label{EQ_one-half}
    \lim \limits_{\delta \to 0^+} \Phi_{\delta}= \frac{1}{2}
  \end{equation}
  locally uniformly on $(-\frac{\pi}{2},\frac{\pi}{2})$.
\end{proposition}

\begin{remark} The two  quantities $\mathcal{R}(\cdot,\St)$ and $\mathcal{R}(\cdot,\St_{\delta})$ occurring in (\ref{EQ_one-half_def}) do have simple explicit expressions. In principle, the third quantity $\mathcal{R}(\cdot,\St^*_{\delta})$ does also have an explicit expression, since the conformal map of the unit disk $\D$ onto $\St^*_{\delta}$ is ``explicitly'' known, see \cite[formula (5.2.15), p.~272]{KopStall1959}.  However, this explicit formula is fairly involved and seems quite unsuitable for obtaining precise information about $\mathcal{R}(\cdot,\St^*_{\delta})$, which is needed for proving Proposition \ref{PR_one-half}. The proof of Proposition~\ref{PR_one-half} below circumvents this difficulty by making use of the strong Markov property for the Green's function, see Sect.\,\protect\ref{SUB_Markov}.
  \end{remark}

In order to prove  Proposition \ref{PR_one-half} we further need several auxiliary lemmas.
\begin{lemma}\label{LM_f(x)f(0)}
  Fix some $\beta\in(-\frac\pi2,\frac\pi2)$, some positive $\alpha_0\le\frac12\big(\frac\pi2-\beta\big)$, and some $\alpha\in(0,\alpha_0)$, and let
\begin{equation*}
  h(x):=\Gf_{\St}\big(i\beta, x+i\left(\tfrac\pi2-\alpha\right)\big), \quad x  \in \R.
\end{equation*}
  Then
\begin{equation}\label{EQ_f(x)f(0)}
h(0) \,\ge \,h(x)\,\ge\,h(0)\,
\frac{1-\cos\alpha_0}{\cosh x-\cos\alpha_0}
\end{equation}
for all~$x\in\Real$.
\end{lemma}
\begin{proof}
Let $z:=x+i(\tfrac\pi2-\alpha)$. Utilizing the formula for the Green's function of $\St$ (see \cite[p.109]{Ransford}), we find
\begin{equation*}
h(x)=\Gf_{\St}\big(i\beta, z\big)=\log\left|\frac{e^z+e^{-i\beta}}{e^z-e^{i\beta}}\right|=\tfrac12\log q(\cosh x,\alpha),\quad q(u,\alpha):=\frac{u-\sin(\beta-\alpha)}{u-\sin(\beta+\alpha)}.
\end{equation*}
The proof of (\ref{EQ_f(x)f(0)}) is now elementary. For convenience we provide the main steps.
As  $q(1,\alpha) \ge q(u,\alpha)>1$ for any $u\ge1$, we see that the left inequality in (\ref{EQ_f(x)f(0)}) holds.
Since $u \mapsto \varrho(u):=\left( u-\sin(\beta+\alpha)\right) \log q(u,\alpha)$ is concave with $\lim_{u \to \infty} \varrho'(u)=0$, the function $\varrho(u)$ is increasing and we have
\begin{align*}
\log q(u,\alpha) \ge \frac{1-\sin(\beta+\alpha)}{u-\sin(\beta+\alpha)}\,\log q(1,\alpha)%
&  > \frac{1-\sin(\beta+\alpha_0)}{u-\sin(\beta+\alpha_0)}\,\log q(1,\alpha)\\ & \ge
 \frac{1-\cos\alpha_0}{u-\cos\alpha_0}\,\log q(1,\alpha).
\end{align*}
The right inequality in~\eqref{EQ_f(x)f(0)} follows easily.
\end{proof}

\begin{lemma}\label{LM_harmonic-measure-in-D}
Let $\mathcal B:=\{e^{i\theta}:\alpha<\theta<\beta\}$ and $\mathcal B':=\{e^{i\theta}:\alpha'<\theta<\beta'\}$. Let $z,\,z'\in\UD$. Then
\begin{equation}\label{EQ_harmonic-measure-in-D}
\big|\Hm_{\UD}(z,\mathcal B)-\Hm_{\UD}(z',\mathcal B')\big|~\le~ \frac{|\alpha'-\alpha|+|\beta'-\beta|+2|z'-z|}{\pi(1-r)},
\end{equation}
where $r:=\max\{|z|,|z'|\}$.
\end{lemma}
\begin{proof}
  Recall that $2\pi\Hm_\UD(0,\cdot)$ coincides with one-dimensional Lebesgue measure on the unit circle $\UC$.
  Hence
    \begin{equation} \label{eq:HMLipschitz}
      \left|\Hm_{\D}(0,\mathcal B)-\Hm_{\D}(0,\mathcal B')\right| \le \frac{|\alpha'-\alpha|+|\beta'-\beta|}{2\pi} \, .
      \end{equation}
To prove~\eqref{EQ_harmonic-measure-in-D} in the case $z'=z\neq0$, it is sufficient to apply (\ref{eq:HMLipschitz}) to the arcs $T(\mathcal B,z)$ and $T(\mathcal B',z)$, where $T(\sigma,z):=(\sigma-z)/(1-\sigma \bar z)$, and take into account that
\begin{equation*}
\left|\frac{\partial T(\sigma,z)}{\partial \sigma}\right|\le\frac{1+|z|}{1-|z|}<\frac{2}{1-r}\quad\text{for any $\sigma\in\UC$}.
\end{equation*}
Finally, to prove~\eqref{EQ_harmonic-measure-in-D} in the general case, we notice that for any $\sigma\in\UC$ and $\theta\in\Real$,
\begin{equation*}
 \left|\,\frac{\partial T(\sigma,z+te^{i\theta})}{\partial t}\Big|_{t=0}\,\right| ~\le~ \left|\frac{\partial T(\sigma,z)}{\partial z}\right|\,+\,\left|\frac{\partial T(\sigma,z)}{\partial \bar z}\right| ~\le~ \frac2{1-r}.
\end{equation*}
Hence, $\big|\arg T(\sigma,z')/T(\sigma,z)\big|\le 2|z'-z|/(1-r)$ for any~$\sigma\in\UC$. As a consequence, \begin{equation*}
 2\pi\big|\Hm_\UD(z',\mathcal B')-\Hm_\UD(z,\mathcal B')\big|\le4|z'-z|/(1-r)
\end{equation*}
and the general case for the estimate~\eqref{EQ_harmonic-measure-in-D} follows immediately.
\end{proof}

\begin{proof}[\proofof{Proposition~\ref{PR_one-half}}]
Applying formula~\eqref{EQ_MarkovConfRad} we immediately get
\begin{align}
\log\frac{\confr{iy}{\St_{\delta}}}{\confr{iy}{\St}}&=\int \limits_{A} \Gf_{\St_{\delta}}(iy,z)\,\Hm_{\St}(iy,\di z),&& \text{where $A:=\left\{z:\Im z=\tfrac\pi2\right\}$},\label{EQ_intA}\\[2ex]
\log\frac{\confr{iy}{\St_{\delta}}}{\confr{iy}{\St^*_\delta}}&=\int \limits_{B\,\cup\, C_\delta}\!\!\! \Gf_{\St_{\delta}}(iy,z)\,\Hm_{\St_{\delta}^*}(iy,\di z), && \text{where $B:=\left\{z: \Im z=\tfrac\pi2,~\Re z<0\right\}$}\label{EQ_intA-delta}\\& &&\text{and $C_\delta:=\left\{iy: \tfrac\pi2 \le  y < \tfrac\pi2+\delta\right\}$.}\notag
\end{align}

Clearly, we may restrict consideration to small $\delta>0$; namely, we will suppose that
\begin{equation*}
 \delta\in(0,\delta_0),\quad\text{where~}~\delta_0:=\min_{y\in K}\tfrac12\big(\tfrac{\pi}2-y\big).
\end{equation*}

Fix temporarily~$y\in K$ and $\delta\in(0,\delta_0)$. By Lemma~\ref{LM_f(x)f(0)} applied with
\begin{equation*}
\beta:=\Big(y-\frac\delta2\Big)\frac\pi{\pi+\delta},\quad \alpha_0:=\delta_0,\quad \alpha:=\delta\frac{\pi}{\pi+\delta},
\end{equation*}
we have that $h(x)=h(x;y,\delta):=\Gf_{\St_{\delta}}(iy,x+i\tfrac\pi2)$ satisfies inequality~\eqref{EQ_f(x)f(0)}.
On the one hand, in combination with the explicit formula from \cite[Table 4.1, p.100]{Ransford}
\begin{equation*}
\Hm_{\St}(iy,\di z)=(2\pi)^{-1}\cos y\,(\cosh x-\sin y)^{-1}\di x\, , \quad z=x+i\pi/2\in A,
\end{equation*}
this allows us to estimate the r.h.s. of~\eqref{EQ_intA} from below:
\begin{eqnarray}
\int \limits_{A} \Gf_{\St_{\delta}}(iy,z)\,\Hm_{\St}(iy,\di z)&\ge&\displaystyle\frac{h(0;y,\delta)}{2\pi}\int \limits_{\Real}
\frac{(1-\cos\delta_0)\sin 2\delta_0}{(\cosh x-\cos\delta_0)(\cosh x+1)}
\,\di x \notag
\\[1.5ex]\label{EQ_rhs-estimate} &=&m(\delta_0)h(0;y,\delta)\quad \text{for all~$\delta\in(0,\delta_0)$ and all $y\in K$,}
\end{eqnarray}
where $m(\delta_0)$ is a positive constant depending only on~$\delta_0$.
On the other hand, in view of (\ref{EQ_f(x)f(0)}), we have
\begin{align}
  \int \limits_{A} \Gf_{\St_{\delta}}(iy,z)\,\Hm_{S}(iy,\di z) & \le \displaystyle h(0;y,\delta) \int\limits_{A} \,\Hm_{\St}(iy,\di z)=h(0;y,\delta) \left( \frac{y}{\pi}+\frac{1}{2} \right) \notag \\ &  \label{EQ_rhs-estimate_new}   \le h(0;y,\delta) \left( 1-\frac{2 \delta_0}{\pi} \right) \, .
\end{align}

The main ingredient of the proof is the following: \\[1ex]
{\bf Claim 1.} As $\delta\to0^+$,
\begin{equation}\label{EQ_conv-of-int-over-B}
 \sup_{y\in K}~\frac{1}{h(0;y,\delta)}\left|\int\limits_{B} \Gf_{\St_{\delta}}(iy,z)\,\Hm_{\St_{\delta}^*}(iy,\di z)-\frac12\int\limits_{A} \Gf_{\St_{\delta}}(iy,z)\,\Hm_{\St}(iy,\di z)\right|~\to~0.
\end{equation}

To prove this claim we first notice that, due to symmetry, one can remove the factor $1/2$ in~\eqref{EQ_conv-of-int-over-B} if the set $A$ in the second integral is replaced by~$B$. Moreover,
 by the monotonicity property of harmonic measure, see e.g. \cite[Corollary~4.3.9 on p.\,102]{Ransford},
\begin{equation*}
 \mu_{y,\delta}\,:=\,\Hm_{\St_{\delta}^*}(iy,\cdot)\big|_B\,-\,\Hm_{\St}(iy,\cdot)\big|_B
\end{equation*}
is a non-negative bounded measure on~$B$. Recall also that $0 < h(x;y,\delta)\le h(0;y,\delta)$ for all $x\in\Real$, see (\ref{EQ_f(x)f(0)}). Therefore, in order to prove~\eqref{EQ_conv-of-int-over-B}, it is sufficient to show that
\begin{equation}\label{EQ_h-measure-of-B-convergence}
\sup_{y\in K}\mu_{y,\delta}(B)\to 0\quad \text{as $\delta\to0^+$.}
\end{equation}
To this end, we will take advantage of conformal invariance of harmonic measure. Denote by  $F_\delta$, $\delta\in(0,\delta_0)$, and $F$ the conformal mappings of $\St_{\delta}^*$ and $\St$, respectively, onto~$\UD$ with the normalization ${F(0)=F_\delta(0)=0}$, ${F'(0)>0}$, ${F'_\delta(0)>0}$. Clearly, ${F(z)=(e^z-1)/(e^z+1)}$ extends holomorphically and injectively to the wider strip ${|\Im z|<\pi}$ and hence  we can write $F_\delta=f_\delta^{-1}\circ F$, where $f_\delta$ is the conformal mapping of~$\UD$ onto the Jordan domain $D_\delta:=F(\St^*_\delta)$ with $f_\delta(0)=0$ and $f'_\delta(0)>0$.

Denote $\mathcal B:=F(B)=\{z \in \partial \D:~\pi/2<\arg z<\pi\}$. By conformal invariance of harmonic measure,
$\Hm_{\St}\big(iy,B\big)=\Hm_\UD\big(F(iy),\mathcal B\big)$ and $\Hm_{\St_{\delta}^*}\big(iy,B\big)=\Hm_\UD\big(f_\delta^{-1}(F(iy)),f_\delta^{-1}(\mathcal B)\big)$.

Note that $\UD\subset D_\delta$ for any $\delta\in(0,\delta_0)$ and that $D_\delta\to \UD$  w.r.t.\,$0$ in the sense of kernel convergence when $\delta\to0^+$. Denote ${z=z(y):=F(iy)}$ and ${z'=z'(y,\delta):=F_\delta(iy)=f^{-1}_\delta(z(y))}$. As~${\delta\to0^+}$, by Carath\'eodory's kernel convergence theorem, we have $z'(y,\delta)\to z(y)$ uniformly w.r.t.~${y\in K}$. Note also that by the Schwarz lemma, for any ${y\in K}$ and any ${\delta\in(0,\delta)}$ the estimate  $|z'(y,\delta)|\le|z(y)|\le r_0:=\max_{y\in K}|F(iy)|<1$ holds.
Therefore, \eqref{EQ_h-measure-of-B-convergence} and hence~\eqref{EQ_conv-of-int-over-B} follow  from Lemmas~\ref{LM_convergence-on-the-boundary-arc} and~\ref{LM_harmonic-measure-in-D}.
\medskip

Taking into account~\eqref{EQ_intA}\,--\,\eqref{EQ_conv-of-int-over-B}, we see that it remains to estimate in~\eqref{EQ_intA-delta}
the part of the integral over~$C_\delta$. We claim that\\[1ex]
{\bf Claim 2.} As $\delta\to0^+$,
\begin{equation}\label{EQ_conv-of-int-over-C}
 \sup \limits_{y\in K}~\frac{1}{h(0;y,\delta)}\left|\,\,\int\limits_{C_\delta} \Gf_{\St_{\delta}}(iy,z)\,\Hm_{\St_{\delta}^*}(iy,\di z)\right|~\to~0.
\end{equation}\medskip

Fix some ${\delta\in(0,\delta_0)}$ and ${y\in K}$. Clearly, $0<\Gf_{\St_{\delta}}(iy,z)\le h(0;y,\delta)$ for any  ${z\in C_\delta}$. Therefore, it is sufficient to show that
\begin{equation}\label{EQ_h-measure-of-C-convergence}
  \sup_{y\in K}\Hm_{\St^*_{\delta}}(i y,C_\delta)\to 0\quad \text{as $\delta\to0^+$.}
\end{equation}

Consider again the mappings $F_\delta$ introduced above. As we have shown in the proof of Claim~1, $|F_\delta(iy)|\le r_0$, where $r_0\in[0,1)$ depends only on the compact set~$K$. Therefore, to prove~\eqref{EQ_h-measure-of-C-convergence} we only need to check that the linear Lebesgue measure of~$F_\delta(C_\delta)$ tends to zero as~$\delta\to0^+$. Suppose this is not the case. Then there exists a sequence $(\delta_n)\subset (0,\delta_0)$ converging to~$0$ and a non-degenerate closed arc $\mathcal C\subset\UC$  such that for each $n\in\Natural$, ${g_n:=F_{\delta_n}^{-1}:\UD\to \St^*_{\delta_n}}$ extends continuously to~$\mathcal C$, with $g_n(\mathcal C)\subset C_{\delta_n}$.

Since $g_n$ is continuous on the compact set $\{\sigma r:\sigma\in\mathcal C,~r\in[0,1]\}$, there exists $r_n\in(0,1)$ with the property that $g_n(\mathcal C_n)$, where $\mathcal C_n:=\{\sigma r_n:\sigma\in\mathcal C\}$, lies in the $1/n$\,-\,neighbourhood of~$C_{\delta_n}$. 
Thus, $\diam g_n(\mathcal{C}_n) \to 0$.
Note that the functions $g_n$ are ``uniformly normal'' in~$\UD$ in the sense that
\begin{equation}\label{EQ_uniformly}
\sup_{z\in\UD}\frac{|g'_n(z)|(1-|z|^2)}{1+|g_n(z)|^2}\,\le\,M\quad\text{for all~$n\in\Natural$}
\end{equation}
and some constant $M>0$ independent of~$n$. Indeed, $\St^*_{\delta}\subset \St^*_{\delta_0}$. Therefore, $g_n=F_{\delta_0}^{-1}\circ\phi_n$, where $\phi_n$ is a holomorphic self-map of~$\UD$. Taking into account that $F_{\delta_0}^{-1}$  is univalent in~$\UD$ and hence normal (see \cite[Lemma~9.3 on p.\,262]{Pombook75}) the desired conclusion~\eqref{EQ_uniformly} follows from the Schwarz\,--\,Pick Lemma.
Now by the No-Koebe-arcs theorem for sequences of holomorphic functions (see \cite[Theorem~9.2 and the subsequent remark on p.\,265]{Pombook75}) it follows that $(g_n)$ converges locally uniformly in $\D$ to a constant, which is impossible because $g_n \to F^{-1}$ by construction. This contradiction proves Claim~2.

\medskip

Now the conclusion of the proposition follows easily from~\eqref{EQ_intA}\,--\,\eqref{EQ_conv-of-int-over-B} and~\eqref{EQ_conv-of-int-over-C}.
\end{proof}

We are now in a position to prove Theorem \ref{thm:main2suff}.

\begin{proof}[\proofof{Theorem \ref{thm:main2suff}}]
  Let $\delta_j(t):=\delta_{\Omega,j}(t)$ and let $K$ be a compact subset of the interval $(-\frac{\pi}{2},\frac{\pi}{2})$. It suffices to prove that there exist $T_j<0$ and $c_j>0$, $j=1,2$,  such that
\begin{equation} \label{eq:sufflowerbound}
  \log \frac{\mathcal{R}\left(t+i y,\Omega\right)}{\mathcal{R}\left(i y,\St\right)} \ge c_j \delta_j(t) \quad \text{ for all } t \le T_j \text{ and all } y \in K \, .
  \end{equation}
 Further, it clearly suffices to consider the case $j=2$. Fix $t \le 0$. Since $\St^*_{\delta_2(t)} +t \subset \Omega$, we have
    \begin{equation*}
     \mathcal{R}\left(t+i y,\Omega\right) \ge \mathcal{R}\left(t+iy,\St^*_{\delta_2(t)}+t\right) = \mathcal{R}\left(i y,\St^*_{\delta_2(t)}\right)
    \quad \text{ for all } y \in \left(-\frac{\pi}{2},\frac{\pi}{2}\right).
    \end{equation*}
    Hence  Proposition \ref{PR_one-half} implies that
    \begin{equation} \label{eq:suffinequ}
      \log  \left[\frac{\mathcal{R}\left(t+i y,\Omega\right)}{\mathcal{R}\left(i y,\St\right)}\right]  \ge \log \left[ \frac{\mathcal{R}\left(i y,\St^*_{\delta_2(t)}\right)}{\mathcal{R}\left(i y,\St\right)}\right] \ge
\Phi_{\delta_2(t)}(y) \cdot  \log \left[\frac{\mathcal{R}\left(iy,\St_{\delta_2(t)}\right)}{\mathcal{R}\left(i y,\St\right)}\right]
    \end{equation}
    for any $t \le 0$ and any $y \in (-\frac{\pi}{2},\frac{\pi}{2})$, where $\Phi_{\delta} $ converges to $1/2$ uniformly on $K$ as $\delta\to 0^+$.
By employing the well-known explicit expression (see e.g. \cite[Example 7.9]{BM2007})
\begin{equation*}
 \mathcal{R}\left(i y,\St_{\delta}\right)=\frac{2}{\Hd_{\St_{\delta}}(i y)}=4\frac{\pi+\delta}{\pi} \cos  \left( \frac{\pi}{2}\cdot  \frac{2 y-\delta}{\pi+\delta} \right)
 \end{equation*}
for the conformal radius of the strip $\St_{\delta}$ it is easily checked that
   \begin{equation*}
    \lim \limits_{\delta\to 0} \frac{1}{\delta} \log \left[ \frac{\mathcal{R} \left( iy,\St_{\delta}\right)}{\mathcal{R} \left( iy,\St \right)}\right]=
   \frac{1}{\pi}+\frac{(\pi+2 y)\tan y}{2 \pi} \in (0,+\infty)
   \end{equation*}
   holds uniformly w.r.t.~$y \in K$. Combining this fact with $\Phi_{\delta}\to 1/2$ uniformly on $K$ as $\delta\to 0^+$ and  $\lim_{t \to -\infty} \delta_2(t)=0$ shows that inequality (\ref{eq:suffinequ})
   implies the estimate (\ref{eq:sufflowerbound}) for $j=2$. As noted above, this concludes the proof of Theorem \ref{thm:main2suff}.
       \end{proof}

\subsection{Necessity of condition~\protect{\eqref{EQ_int-mainthrm2-comvergence}}}\label{SS_necessity}
In this subsection, we prove the necessity of the condition~\protect{\eqref{EQ_int-mainthrm2-comvergence}}.
The proof relies on the fact that the integral in \eqref{EQ-via-conformal-radius} can be estimated from above by the integral over the function $\delta_{\Omega}$. This is the content of the following theorem.

\begin{theorem} \label{THM:main2necess}
  Let $\Omega$ be a proper subdomain of $\C$ which is starlike at infinity and such that the standard strip $\St=\{z \in \C \, : \, |\Im z| <\frac{\pi}{2}\}$ is a maximal horizontal strip contained in $\Omega$. Then for any compact set $K \subset (-\frac{\pi}{2},\frac{\pi}{2})$ there are constants $c>0$ and $C>0$ such that for any $y \in K$,
\begin{equation} \label{EQ:main2necess}
\int\limits_{-\infty}^{0}  \log \frac{\mathcal{R}(t+iy,\Omega)}{\mathcal{R}(t+iy,\St)} \, \di t  \leq C\,+\, 2c\! \int\limits_{-\infty}^{0}  \delta_{\Omega} (t) \, \di t  .
  \end{equation}
\end{theorem}

Hence Theorem \ref{THM:main2necess} in conjunction with Lemma \ref{lem:WarRodin} implies that the integral of the ratio of conformal radii in \eqref{EQ:main2necess} converges, when the domain $\Omega$ is conformal at $-\infty$ w.r.t. $\St$.

In order to prove Theorem~\ref{THM:main2necess}, we construct a family of two-slit domains $D_0(t) $ that all contain the Koenigs domain~$\Omega$. For the sake of brevity, denote
\begin{equation*}
 \delta:=\delta_{\Omega}(1+{t}/{2}),\quad \delta_j:=\delta_{\Omega,j}(1+{t}/{2}),~ j=1,2,\quad t\le0.
 \end{equation*}
 Due to $\Omega$ being starlike at infinity, it is easy to see that
\begin{equation*}
D_0(t):=\C \setminus \left\{z \in \C :\, \Re z \leq 1+\tfrac{t}{2}~ \text{and}~ \Im z \in \{ -\tfrac{\pi}{2}-\delta_1  , \tfrac{\pi}{2}+\delta_2\} \right\},\quad t\le 0,
\end{equation*}
indeed contains $\Omega$.
Let $y\in (-\frac{\pi}{2}, \frac{\pi}{2})$. Then
\begin{equation}\label{EQ:IntConvergence}
    \log \frac{\confr{t+iy}{\Omega}}{\confr{t+iy}{\St}} \leq  \log \frac{\confr{t+iy}{D_0(t)}}{\confr{t+iy}{\St}}.
\end{equation}

Furthermore, let $S(t):= \St \left(-\frac{\pi}{2}-\delta_1 , \frac{\pi}{2}+\delta_2  \right)$, the maximal strip between the two slits of the domain $D_0(t)$, and let
\begin{equation*}
D(t) := D_0(t)-\big(1+\tfrac{t}2\big)=\C \setminus \left\{z \in \C :\, \Re z \leq 0~ \text{and}~ \Im z \in \{ -\tfrac{\pi}{2}-\delta_1  , \tfrac{\pi}{2}+\delta_2\} \right\}.
\end{equation*}

Bearing in mind that the hyperbolic metric in a strip remains invariant under horizontal translations, it follows that
\begin{equation}\label{EQ:Upperbound}
    \log \frac{\confr{t+iy}{\Omega}}{\confr{t+iy}{\St}} \leq \log \frac{\confr{\frac{t}{2}-1+iy}{D(t)}}{\confr{\frac{t}{2}-1+iy}{S(t)}} + \log \frac{\confr{iy}{S(t)}}{\confr{iy}{\St}}.
\end{equation}

The idea for the proof of Theorem~\ref{THM:main2necess} is to analyze the asymptotic behavior of the summands in \eqref{EQ:Upperbound}. This task is rather elementary for the second term, while the first term requires a bit more delicate work making use of an estimate of the Green's function, which we establish in the following lemma. Denote by $D_*$  the ``standard'' two-slit domain $\mathbb{C} \setminus \{w \in \mathbb{C}: {\Re{}w\leq0},~ {|\Im{} w\,|\, =\, {\pi}/{2}}\}$.

\begin{lemma}\label{LM_GF-est}
For all~$w\in\St$ with $\Re w < 0$ and all $s\in D_*\cap\partial\St$, we have
\begin{equation}\label{EQ_GF-est}
\Gf_{D_*}(w,s)\le \log\frac{1+e^{\Re w}}{1-e^{\Re w}}.
\end{equation}
\end{lemma}
\begin{proof}
Using the explicit formula for the Green's function of~$\St$, we have
\begin{eqnarray}\notag
 \Gf_{\St}(z_1,z_2)&=&   \dfrac12\log\dfrac{\cosh(x_2-x_1)+\cos(y_2+y_1)}{\cosh(x_2-x_1)+\cos(y_2+y_1)}\\
 \label{EQ_GF-est0}
                   &\le& \log\dfrac{1+e^{-|x_2-x_1|}}{1-e^{-|x_2-x_1|}},\quad x_k:=\Re z_k,~y_k:=\Im z_k,~k=1,2,
\end{eqnarray}
for any pair of points $z_1,z_2\in\St$.

In order to apply the above formula for estimating $\Gf_{D_*}$, we notice that $g(z):=z+(1+e^{2z})/2$  maps $\St$ conformally onto $D_*$. Indeed, $f(\zeta):=g(\log \zeta)=\log(\zeta)+(1+\zeta^2)/{2}$
is the Schwarz\,--\,Christoffel map of the right-half plane onto $D_*$ with the normalization $f(0)=\infty$, $f(\pm i)=\pm i {\pi}/{2}$, as follows from the equality
\begin{equation*}
 \frac{f''(\zeta)}{f'(\zeta)}=\frac{1}{\zeta-i}+\frac{1}{\zeta+i}-\frac{1}{\zeta},\quad \Re\zeta>0.
\end{equation*}

Let $w\in U:=\{w\in\St\colon\Re w<0\}$. It is easy to see that $g(U)\supset U$. Therefore, ${z:=g^{-1}(w)\in U}$. It follows that
\begin{equation}\label{EQ_Re-Re}
 \Re g^{-1}(w)=\Re w-(1+\Re e^{2z})/2 < \Re w\quad\text{for all~$w\in U$}.
\end{equation}

Moreover, it is easy to see that $g\big(\{w\in\St\colon \Re w>0\}\big)$ contains the set ${A:=D_*\cap\partial\St}$. Hence,
\begin{equation}\label{EQ_u_0}
 \Re g^{-1}(s)>0\qquad\text{for all~$s\in A$}.
 \end{equation}
Combining now \eqref{EQ_GF-est0}, \eqref{EQ_Re-Re} and~\eqref{EQ_u_0}, we obtain
\begin{equation*}
\Gf_{D_*}(w,s)=\Gf_{\St}\big(g^{-1}(w),g^{-1}(s)\big)\le\log\frac{1+e^{\Re w}}{1-e^{\Re w}}
\end{equation*}
for any ${w\in U}$ and any ${s\in A}$, as desired.
\end{proof}

\begin{lemma}\label{LM_two-slit-VS-strip_est}
For any $w\in U:=\{{w\in\St}\colon {\Re w<0}\}$, we have
\begin{equation}\label{EQ_two-slit-VS-strip_est}
 \log \frac{\confr{w}{D_*}}{\confr{w}{\St}}\le \log\frac{1+e^{\Re w}}{1-e^{\Re w}}.
\end{equation}
\end{lemma}
\begin{proof}
According to~\eqref{EQ_MarkovConfRad}, for any~$w\in\St$, the l.h.s. of~\eqref{EQ_two-slit-VS-strip_est} equals
\begin{equation*}
\int_A \Gf_{D_*}(w,s) \,\Hm_{\St}(w,\di s),\quad A:=D_*\cap\partial\St,
\end{equation*}
and inequality~\eqref{EQ_two-slit-VS-strip_est} follows immediately from~\eqref{EQ_GF-est}, taken into account that $\Hm_{\St}(w,\cdot)$ is a probability measure.
\end{proof}

\begin{proof}[\proofof{Theorem~\ref{THM:main2necess}}]
Using the explicit formula for the conformal radius of a strip, it is easy to see that there exists a constant $c>0$, depending only on the compact set~$K\subset\big(-\tfrac\pi2,\tfrac\pi2\big)$, such that
\begin{equation}\label{EQ_simple}
   \log \frac{\confr{iy}{S(t)}}{\confr{iy}{\St}}\le c\delta=c\,\delta_\Omega\big(1+t/2\big)\quad\text{for all $t\le0$ and all $y\in K$.}
\end{equation}
Moreover, recall that the function $\delta_\Omega$ is finite and monotonic on~$(-\infty,1]$. In particular, it is integrable on~$[0,1]$.
Therefore, in view of inequality~\eqref{EQ:Upperbound}, it remains to show that
\begin{equation}\label{EQ_remains-to-prove}
\sup_{y\in K}~\int\limits_{-\infty}^0 \log \frac{\confr{\frac{t}{2}-1+iy}{D(t)}}{\confr{\frac{t}{2}-1+iy}{S(t)}} \, \di t ~<~+\infty.
\end{equation}
The linear function
\begin{equation*}
 F_t(z): = \frac{\pi}{\pi + \delta_1  + \delta_2 } \left(z +i\,\tfrac{ \delta_1 - \delta_2}{2} \right),
\end{equation*}
maps the strip $S(t)$ conformally onto the standard strip $\St$ and the two-slit domain~$D(t)$ onto~$D_*$. Therefore, the integrand in~\eqref{EQ_remains-to-prove} equals
\begin{equation*}
  \log \frac{\confr{w(t)}{D_*}}{\confr{w(t)}{\St}},\quad \text{where~}~w(t):=F_t\big(\tfrac t2-1+iy\big).
\end{equation*}
By monotonicity of~$\delta_{\Omega,j}$, $j=1,2$,
\begin{equation*}
 \Re w(t)\le\dfrac{\pi}{\pi+\delta_{\Omega,1}(1)+\delta_{\Omega,2}(1)}\Big(\frac t2-1\Big)\quad\text{for any $t\le0$},
\end{equation*}
and we are done because now~\eqref{EQ_remains-to-prove} follows from Lemma~\ref{LM_two-slit-VS-strip_est}.
\end{proof}

\begin{proof}[\proofof{necessity in Theorem~\ref{thm:main2}}]
Suppose that $\Omega$ is conformal at~$-\infty$ w.r.t.~$S$. By Lemma \ref{lem:WarRodin}, $\delta_{\Omega}$ is integrable over $(-\infty,0]$.
Hence, Theorem \ref{THM:main2necess} shows that condition~\eqref{EQ_int-mainthrm2-comvergence} holds.
\end{proof}

\subsection{Completion of the proof of Theorem \ref{thm:main2}: locally uniform convergence of  (\ref{EQ_int-mainthrm2-comvergence})}\label{SUB:compl-of-the-proof}

We are left to prove that if~\eqref{EQ-via-conformal-radius} holds for one point ${w_0\in S}$, then it holds for all~${w_0 \in S}$ and the integral in~\eqref{EQ-via-conformal-radius} converges locally uniformly w.r.t.~${w_0\in S}$. Clearly, we may assume~${S=\St}$.

Let $S(t)$ and $D(t)$ be defined as in Subsection~\ref{SS_necessity}. For $t\le0$ and ${y\in\big(-\tfrac\pi2, \tfrac\pi2\big)}$ we denote
\begin{equation*}
F_1(t,y):=\log \frac{\confr{\frac{t}{2}-1+iy}{D(t)}}{\confr{\frac{t}{2}-1+iy}{S(t)}},~{}~ F_2(t,y):=\log \frac{\confr{iy}{S(t)}}{\confr{iy}{\St}},~{}~ J_k(y):=\int\limits^0_{-\infty}\! F_k(t,y)\di t,~k=1,2.
\end{equation*}

The function $\delta_\Omega$ is monotonic on~$(-\infty,1]$ and hence integrable on any compact interval contained in~${(-\infty,1]}$. Therefore, from Theorem~\ref{thm:main2suff} it follows that if~\eqref{EQ-via-conformal-radius} holds for at least one point ${w_0\in\St}$, then the function $\delta_\Omega$ is integrable on~${(-\infty,1]}$.

Thanks to inequality~\eqref{EQ_simple}, the integrability of $\delta_\Omega$ on~${(-\infty,1]}$ implies that the integral $J_2(y)$ converges uniformly w.r.t.~$y$ on any compact interval~${[y_1,y_2]\subset \big(-\tfrac\pi2, \tfrac\pi2\big)}$. Moreover, the argument used in the proof of Theorem~\ref{THM:main2necess} shows that the integral~$J_1(y)$ converges uniformly on the whole interval~$\big(-\tfrac\pi2, \tfrac\pi2\big)$.

Now let $w_0:=x+iy\in\St$. By inequality~\eqref{EQ:Upperbound}, for all $t\le\min\{-x,\,0\}$ we have
\begin{equation*}
  0\le\log \frac{\confr{t+w_0}{\Omega}}{\confr{t+w_0}{\St}}\le F_1(t+x,y)+F_2(t+x,y).
\end{equation*}

Thus, for each rectangle $R:=[x_1,x_2] \times [y_1,y_2]$ contained in~$\St$, the integral in~\eqref{EQ-via-conformal-radius}
converges uniformly w.r.t. $w_0\in R$. This completes the proof.\qed

\section{The angular derivative problem for parabolic petals}\label{SEC:parabolic_petals}
In this section we study conformality of parabolic petals at the Denjoy\,--\,Wolff point. An important role is played by the so-called hyperbolic step $q(z,s):=\lim_{t\to+\infty} \hdist_{\UD}\big(\phi_{t+s}(z),\phi_t(z)\big)$, where $\hdist_{\UD}$ stands for the hyperbolic distance in~$\UD$. If $q(z,s)=0$ for some ${z\in\UD}$ and~$s>0$, then $q\equiv0$ on ${\UD\times(0,+\infty)}$ and the one-parameter semigroup $(\phi_t)$ is said to be of \textsl{zero hyperbolic step}. Otherwise, i.e. if $q(z,s)>0$ for some (and hence all) ${z\in\UD}$ and ${s>0}$, we say that $(\phi_t)$ is of \textsl{positive hyperbolic step}.

\begin{theorem} \label{thm:parabolic}
Let $\Delta$ be a parabolic petal of a one-parameter semigroup~$(\phi_t)$ with Denjoy\,--\,Wolff point~$\tau$ and associated Koenigs function~$h$. The following statements hold:
\begin{itemize}
\item[(A)] The petal $\Delta$ is conformal at~$\tau$ w.r.t.~$\UD$ if and only if the angular limit \begin{equation*}
      L:=\smash[b]{\anglim_{z\to\tau}(z-\tau)h(z)}
    \end{equation*}
       is finite.\smallskip
\item[(B)] If $(\phi_t)$ is of positive hyperbolic step, then $\Delta$ is conformal at~$\tau$ w.r.t.~$\UD$
(and hence~$L\neq\infty$).
\end{itemize}
\end{theorem}
\begin{remark}
The angular limit $L$ in statement~(A) above was previously considered by Contreras, D\'\i{}az-Madrigal and Pommerenke in a more general context of discrete iteration, see \cite[Theorems~4.1 and~6.2]{CDP2010}. Moreover, this limit is directly related to the conformality problem considered by Betsakos~\cite[Theorem~3]{Bet2016} and Karamanlis~\cite[Theorem~2]{NK2019} for parabolic one-parameter semigroups of positive hyperbolic step.
As we will see in the proof of Theorem~\ref{thm:parabolic}, the presence of a parabolic petal ensures that the angular limit $L$ exists and does not vanish, but it can be infinite. In fact, as illustrated by Example~\ref{EX_parab} given after the proof, for semigroups of zero hyperbolic step both cases $L\in\C\setminus\{0\}$ and ${L=\infty}$ are possible.
\end{remark}

\begin{proof}[\proofof{Theorem~\ref{thm:parabolic}}]
First of all, recall that the existence of a parabolic petal~$\Delta$ implies that the one-parameter semigroup $(\phi_t)$ is parabolic. Moreover, the image of~$\Delta$ w.r.t. the Koenigs function $h$  of~$(\phi_t)$ is a half-plane bounded by a line parallel to~$\Real$.

If $(\phi_t)$ is of positive hyperbolic step, then $\Omega:=h(\UD)$ is contained in some half-plane~$U$ with $\partial U$ parallel to~$\Real$, see \cite[Theorem~9.3.5]{BCD-Book}.  Let $T$ be a linear-fractional transformation of~$\UD$ onto~$U$ with ${T(1)=\infty}$. Choose a number ${c\in\Complex}$ such that $w\mapsto w+c$ maps $U$ onto $h(\Delta)$. Then
\begin{equation*}
 \varphi(z):=h^{-1}\big(T(z)+c\big),\quad z\in\UD,
\end{equation*}
is a conformal mapping of~$\UD$ onto~$\Delta$ with $\varphi(1)=\tau$ in the sense of the angular limit.
Further, denote $\psi:=T^{-1}\circ h:\UD\to\UD$. Since $\psi\circ\varphi=T^{-1}\circ(T+c)$ has a regular contact point at~$\zeta=1$, by the Chain Rule for angular derivatives, see e.g. \cite[Lemma~2]{CDP2006} or \cite[Lemma~5.1]{Gum2017}, $\varphi$ also has a regular contact point at~${\zeta=1}$, which means that $\Delta$ is conformal at~$\tau$ w.r.t.~$\UD$. This proves statement~(B).

To prove~(A),  fix some $w_0\in\C\setminus\Omega$ and denote $g:=1/(h-w_0)$.  Further denote by~$D$ the image of the half-plane $h(\Delta)$ w.r.t. the map ${w\mapsto 1/(w-w_0)}$. Clearly, $D$ is a disk contained in~$g(\UD)$, with $0\in\partial D$. Let $f$ be the linear function mapping $\UD$ onto~$D$ and normalized by ${f(\tau)=0}$. Then for a suitably chosen circular arc $\Gamma\subset\UD$ with the end-point at~$\tau$, $g^{-1}\big(f(\Gamma)\big)$ is the forward orbit of some point ${z_0\in\Delta}$ w.r.t. the semigroup~$(\phi_t)$. In particular, ${g^{-1}(f(z))\to\tau}$ as ${\Gamma\ni z\to\tau}$. Therefore,
by the Comparison Theorem, see e.g. \cite[Theorem~10.6 on p.\,307]{Pombook75}, there exists a finite angular derivative $g'(\tau)$. In turn, this implies that
$L:=\anglim_{z\to\tau}(z-\tau)h(z)$ exists, with ${L=1/g'(\tau)\in\ComplexE\setminus\{0\}}$.

Note that $\varphi(\zeta):=g^{-1}\big(f(\tau\zeta)\big)$  is a conformal mapping of $\UD$ onto~$\Delta$ and that by Lindel\"of's Theorem,  see e.g. \cite[Theorem~9.3 on p.\,268]{Pombook75}, ${\anglim_{\zeta\to\tau}\varphi(\zeta)=\tau}$. Recall that since ${\varphi\in\Hol(\UD,\UD)}$ and since ${\zeta=1}$ is a contact point for~$\varphi$, the angular derivative $\varphi'(1):={\anglim_{\zeta\to1}(\varphi(\zeta)-\tau)/(\zeta-1)}$ does exist, finite or infinite, and ${\varphi'(1)\neq0}$.
The following simple calculation
\begin{equation*}
\tau f'(\tau)=(g\circ \varphi)'(1)=
  \lim_{(0,1)\ni x\to1}\left(\frac{g\big(\varphi(x)\big)}{\varphi(x)-\tau}\cdot\frac{\varphi(x)-\tau}{x-1}\right)
\end{equation*}
shows that $R(z):=g(z)/(z-\tau)$, $z\in\UD$, has at ${z=\tau}$ a finite asymptotic value~$\tau f'(\tau)/\varphi'(1)$. Taking into account that $g$ is univalent and does not vanish in~$\UD$ and arguing as in \cite[proof of Theorem~10.5, pp.\,305--306]{Pombook75}, we see that $g'(\tau)=\anglim_{z\to\tau}R(z)=\tau f'(\tau)/\varphi'(1)$. It follows that ${L=1/g'(\tau)\neq\infty}$ if and only if~$\varphi'(1)\neq\infty$. This completes the proof of~(A).
\end{proof}

\begin{remark}
Note that the centre of the disk~$D$ considered in the above proof lies on the imaginary axis. It follows that under the hypothesis of Theorem~\ref{thm:parabolic}, if $L\neq\infty$ then $\Re(\overline\tau L)=0$.
\end{remark}
\begin{example}\label{EX_parab}
Let $h_1(z):=z/(1-z)^2$ be the classical Koebe function and let
\begin{equation*}
 h_2(z):=w(z)-i\sqrt{w(z)},\quad \text{where $~{w(z):=i(1+z)/(1-z)}$, $~{z\in\UD}$,}
\end{equation*}
and $\sqrt{w}$ stands for the branch of the square root that maps the upper half-plane onto the first quadrant. It is not difficult to see that $\Re(1-z)^2h_k'(z)>0$ for all~${z\in\UD}$, ${k=1,2}$. Therefore, see e.g. \cite[Theorem~9.4.11, p.\,257]{BCD-Book}, $h_k$'s are univalent in~$\UD$ and the formula $\phi^k_t:=h_k^{-1}\circ(h_k+t)$, $t\ge0$, $k=1,2$, defines two parabolic one-parameter semigroups with the DW-point at~${\tau=1}$. Note that $h_1(\UD)={\C\setminus(-\infty,-\tfrac14]}$ and $h_2(\UD)={\{\zeta:\Re\zeta>f(\Im \zeta)\}}$, where ${f(\eta):=-\infty}$ if ${\eta>0}$, ${f(0):=\tfrac14}$, and ${f(\eta):=\eta^2}$ if ${\eta<0}$. Since the image domains $h_k(\UD)$ are not contained in any half-plane, both semigroups~$(\phi_t^k)$, ${k=1,2}$, are of zero hyperbolic step.
 Clearly, ${\lim_{z\to1}(z-1)h_1(z)=\infty}$.  The corresponding semigroup $(\phi_t^1)$ has two non-conformal parabolic petals ${\{z\in\UD:\pm\Im z>0\}}$. At the same time ${\anglim_{z\to1}(z-1)h_2(z)=-2i\neq\infty}$. Since ${\{\zeta:\Im\zeta>0\}}\subset h_2(\UD)$, the semigroup $(\phi^2_t)$ has a parabolic petal. By Theorem~\ref{thm:parabolic}\,(A),  this parabolic petal is conformal at~$\tau$ w.r.t.~$\UD$.
\end{example}

\section{Concluding remarks and open questions}\label{SEC:Remarks}

\subsection{Rate of convergence of regular backward orbits}\label{SUB:rate-of-conv}
If $\sigma\in\UC$ is a repulsive fixed point of a one-parameter semigroup of~$(\phi_t)$, then as it was proved in~\cite[Proposition~4.20]{Bracci_et_al2019},
\begin{equation*}
 \lim_{t\to-\infty}\frac1t\log|\phi_t(z)-\sigma|=\lambda:=G'(\sigma)\quad \text{for any $z\in\Delta(\sigma)$,}
\end{equation*}
where $G$ stands for the infinitesimal generator of~$(\phi_t)$. It is therefore natural to ask whether the limit
\begin{equation}\label{EQ_lim-rate}
 C(\sigma,z):=\lim_{t\to-\infty}e^{-\lambda t}|\phi_t(z)-\sigma|
\end{equation}
does exist for~$z\in\Delta(\sigma)$. The answer is immediate if we recall that $(\phi_t)$ admits at~$\sigma$ a pre-model ${(\UH,\psi,z\mapsto e^{\lambda t}z)}$, where $\psi$ is a conformal mapping of~$\UH$ onto~$\Delta(\sigma)$ with ${\psi(0)=\sigma}$; see Remark~\ref{RM_pre-model}. We have $\phi_t(z)={\psi\big(e^{\lambda t}\psi^{-1}(z)\big)}$  for any ${z\in\Delta(\sigma)}$ and any~${t\in\Real}$. Hence for all~${z\in\Delta(\sigma)}$,
\begin{equation*}
C(\sigma,z)=\lim_{t\to-\infty}e^{-\lambda t}|\psi\big(e^{\lambda t}\psi^{-1}(z)\big)-\sigma|=|\psi^{-1}(z)\psi'(0)|.
\end{equation*}

If the hyperbolic petal~$\Delta(\sigma)$ is conformal, then by Remark~\ref{RM_not-depend-on-normalization}, $\psi^{\prime}(0) \in \mathbb{C}\setminus \{0\}$ and hence the limit~\eqref{EQ_lim-rate} exists finitely and does not vanish for all points~$z$ in $\Delta(\sigma)$. If $\Delta(\sigma)$ is not conformal, then $\psi'(0)=\infty$ and hence $C(z,\sigma)=+\infty$ for all~${z\in\Delta(\sigma)}$.

Thus, for $z_0\in\Delta(\sigma)$ condition~\eqref{EQ_int-mainthrm1-comvergence} in Theorem~\ref{thm:main1} is equivalent to having ${C(z_0,\sigma)\in(0,+\infty)}$. Since the backward orbits in~$\Delta(\sigma)$ converge to~$\sigma$ non-tangentially, see Remark~\ref{RM_backward-orbits-non-tangent}, $C(z_0,\sigma)\in{(0,+\infty)}$ if and only if $\tfrac12\log\big(e^{-\lambda t}(1-|\phi_t(z_0)|^2)\big)$ tends to a finite limit  as ${t\to-\infty}$. In turn, since by Remark~\ref{RM_for-negative-ts}, in~$\Delta(\sigma)$ the ODE~\eqref{EQ_ODE} holds for all~${t\in\Real}$, the latter condition can be restated as the convergence of the integral
\begin{equation}\label{EQ_int-J}
J(z_0):=\int\limits_{-\infty}^0\left(\frac{\lambda}2+\frac{\Re\big[G(\phi_t(z_0))\overline{\phi_t(z_0)}\,\,\big]} {1-|\phi_t(z_0)|^2}\right)\di t.
\end{equation}
At the same time, combining~\eqref{EQ_ODE} and~\eqref{EQ_PDE}, we get $\phi_t'(z_0)={G\big(\phi_t(z_0)\big)/G(z_0)}$. It follows that condition~\eqref{EQ_int-mainthrm1-comvergence} in Theorem~\ref{thm:main1} is equivalent to the convergence of
\begin{equation}\label{EQ_int-I}
I(z_0):=\int\limits_{-\infty}^0\left(\frac{A(z_0)}{2}-\frac{\,\big|G(\phi_t(z_0))\big|\,} {1-|\phi_t(z_0)|^2}\right)\di t,
\end{equation}
where $A(z_0):=\lim\limits_{t\to-\infty}\big|G(\phi_t(z_0))\big| \, \Hd_\UD\big(\phi_t(z_0)\big) =|G(z_0)|\,\Hd_{\Delta(\sigma)}(z_0)$.

\subsection{Boundary behaviour of the Koenigs function}
For a non-elliptic one-parameter semigroup $(\phi_t)$, the conformality of hyperbolic petals is related to the boundary behaviour of the Koenigs function~$h$. Let ${\sigma\in\UC}$ be a repulsive fixed point of~$(\phi_t)$. Since $S:=h\big(\Delta(\sigma)\big)$ is a maximal strip in ${\Omega:=h(\UD)}$, by a result of Betsakos~\cite{Bet2018}, the angular limit
\begin{equation*}
\nu:=\anglim_{z\to\sigma}(z-\sigma)h'(z)~\in~(0,+\infty)
\end{equation*}
exists and equals the width of the strip~$S$ divided by~$\pi$. (Since ${h'=1/G}$, the existence of the above limit follows also from \cite[Theorem~1]{CDP2006}.)

For a suitable $b\in\Complex$, the function $\psi(\zeta):=h^{-1}\big(b+\nu\log \zeta)$ maps the right half-plane $\UH$ onto~$\Delta(\sigma)$. Hence, the conformality of the hyperbolic petal $\Delta(\sigma)$ is equivalent to ${\psi'(0)\neq\infty}$. Thanks to the isogonality property~\eqref{EQ_isogonality}, one can use the change of variables $\zeta:=\psi^{-1}(z)$ to obtain
\begin{equation}\label{EQ_limit-for-h}
\anglim_{z\to\sigma}\big[h(z)-\nu\log(1-\overline\sigma z)\big]=b\,-\,\nu\,\log\big(-\overline\sigma\psi'(0)\big).
\end{equation}
The above limit exists, and is finite or infinite, because $\psi'(0)$ exists with $-\overline\sigma\psi'(0)\in(0,+\infty)\cup\{+\infty\}$.
Thus, $\Delta(\sigma)$ is conformal if and only if the limit in the l.h.s. of~\eqref{EQ_limit-for-h} is \textit{finite}.

\subsection{Semigroups with symmetry w.r.t. the real line}\label{SUB:symm-case}
Consider a one-parameter semigroup $(\phi_t)$ in~$\UD$ with a repulsive fixed point at~${\sigma=-1}$ and such that ${\phi_t\big((-1,1)\big)\subset(-1,1)}$ for all~${t\ge0}$. Fix some $z_0\in\Delta(-1)\cap\Real$. In this rather special case, Theorem~\ref{thm:main1} (excluding the part concerning uniformity of convergence) admits a simple proof based on the following elementary lemma.
\begin{lemma}\label{LM_elementary}
Let $f$ be a function continuous on~$[0,+\infty)$ and of class $C^2$ on~$(0,+\infty)$. Suppose that
\begin{itemize}
\item[(i)] $f(0)=0$ and $f'(x)>0$ for all $x\in(0,+\infty)$;\medskip
\item[(ii)] $g(x):=f(x)/\big(xf'(x)\big)\ge1$ for all~$x\in(0,+\infty)$.
\end{itemize}
Fix $\xi_0>0$.
If $f'(x)\to A$ as ${x\to+\infty}$ for some $A\in(0,+\infty)$, then
\begin{equation}\label{EQ_I}
 I:=\int\limits_0^{\xi_0}~\frac{\log g(x)}{x}\,\di x\,<\,+\infty.
\end{equation}
If $f(x)/x\to+\infty$ as ${x\to+\infty}$, then ${I=+\infty}$.
\end{lemma}
\begin{proof}
Suppose that $f'(x)\to A\in(0,+\infty)$ as ${x\to+\infty}$. Then $f(x)/x\to A$ and $g(x)\to1$ as ${x\to+\infty}$.
Therefore, taking into account~(ii) we see that the improper integral~$I$ converges because
\begin{equation*}
\int\limits_0^{\xi_0}\left(1-\frac1{g(x)}\right)\frac{\di x}{x}=\int\limits_0^{\xi_0}\left(\frac1x-\frac{f'(x)}{f(x)}\right)\di x=\lim_{a\to0^+}\log\frac{x}{f(x)}\Bigg|_{x=a}^{x=\xi_0}\!\!\in~\Real.
\end{equation*}

Similarly, if $f(x)/x\to+\infty$ as ${x\to+\infty}$, then
\begin{equation*}
I~\ge~ \int\limits_0^{\xi_0}\left(1-\frac1{g(x)}\right)\frac{\di x}{x}=\lim_{a\to0^+}\log\frac{x}{f(x)}\Bigg|_{x=a}^{x=\xi_0}\!\!\!=~+\infty. \qedhere
\end{equation*}
\end{proof}
\begin{proof}[\proofof{Theorem~\ref{thm:main1} in the symmetric case}]
Passing from the unit disk~$\UD$ to~$\UH$ with the help of the Cayley map~${H(z):=(1+z)/(1-z)}$, we get a one-parameter semigroup $(\varphi_t)$ in~$\UH$ with a repulsive fixed point at~$0$ and such that ${\varphi_t\big((0,+\infty)\big)}\subset{(0,+\infty)}$ for all~${t\ge0}$. It follows that the hyperbolic petal~$D$ of $(\varphi_t)$ with the $\alpha$-point at~$0$ is symmetric w.r.t. the real line. Therefore, for all $t\in\Real$, all~$z\in D$, and a suitable ${\lambda>0}$, we have
\begin{equation*}
\varphi_t(z)=f\big(e^{\lambda t} f^{-1}(z)\big),
\end{equation*}
where $f$ is a conformal mapping of~$\UH$ onto $D$ with ${f(0)=0}$ and ${f\big((0,+\infty)\big)}\subset{(0,+\infty)}$. As a result, for ${z_0\in\Delta(-1)\cap\Real}$, the integral in~\eqref{EQ_int-mainthrm1-comvergence} equals
\begin{equation}\label{EQ_integral-in-symm-case}
\int\limits_{-\infty}^0\log\frac{f(e^{\lambda t}\xi_0)}{e^{\lambda t}\xi_0 f'(e^{\lambda t}\xi_0)}\,\di t =\frac{1}{\lambda} \int\limits_0^{\xi_0}\left(\log\frac{f(x)}{x f'(x)}\right)\frac{\di x}{x},
\end{equation}
where $\xi_0:=f^{-1}\big(H(z_0)\big)$.

Applying the Schwarz\,--\,Pick Lemma, see e.g. \cite[Theorem 6.4]{BM2007} to ${f:\UH\to\UH}$, it is easy to see that the restriction of $f$ to ${[0,+\infty)}$ satisfies the hypothesis of Lemma~\ref{LM_elementary}. If $\Delta(-1)$ is conformal, then $f'(x)\to f'(0)\in(0,+\infty)$ as $(0,+\infty)\ni {x\to0}$ and hence by Lemma~\ref{LM_elementary}, the integral~\eqref{EQ_integral-in-symm-case} converges.

Similarly, if $\Delta(-1)$ is not conformal, then $f(x)/x\to+\infty$ as $(0,+\infty)\ni {x\to0}$. In this case, Lemma~\ref{LM_elementary} guarantees that the integral~\eqref{EQ_integral-in-symm-case} diverges.
\end{proof}

The above proof of Theorem~\ref{thm:main1} for the symmetric case is based on the observation that $(\phi_t)_{t\in\Real}$ is a one-parameter group of hyperbolic automorphisms of~$\Delta(\sigma)$. In contrast to our proof for the general case, the fact that for ${t\ge0}$, $\phi_t$'s are well-defined holomorphic functions in the whole unit disk is not essential for the proof in the symmetric case.

Attempting to adapt the method used in this section to the general case, instead of the integral~\eqref{EQ_I} one would need to consider
\begin{equation}\label{EQ_I1}
I_1(\theta):=\int\limits_0^{\rho_0}\left(\log\frac{\Re f(\rho e^{i\theta})}{\big|f'(\rho e^{i\theta})\big|\,\rho\cos\theta}\right)\,\frac{\di \rho}{\rho},\quad \theta\in(-\pi/2,\pi/2),
\end{equation}
where ${\rho_0>0}$ is fixed. Since $f$ is isogonal at~$0$, $\Re f(\rho e^{i\theta})/|f(\rho e^{i\theta})|\to\cos\theta$ as $\rho\to0^+$. Hence it would be also reasonable to consider
\begin{equation}\label{EQ_I2}
I_2(\theta):=\int\limits_0^{\xi_0}\log\left|\frac{f(\rho e^{i\theta})}{\rho f'(\rho e^{i\theta})}\right|\,\frac{\di \rho}{\rho}.
\end{equation}
Note that the argument of $\log|\cdot|$ in~\eqref{EQ_I2} tends to~$1$ as $\rho\to0^+$, because by~\cite[Proposition~4.11 on p.\,81]{Pommerenke:BB}, the conformal map ${f\circ H}$, $H(z):={(1-z)/(1+z)}$, satisfies at ${\zeta=1}$ the Visser\,--\,Ostrowski condition.

\subsection{Hyperbolic length of backward orbits}
For a hyperbolic domain $D \subset \C$ and a rectifiable curve $\gamma : [0,T] \to D$ we denote by
\begin{equation*}
 \ell_D(\gamma):=\int \limits_{\gamma} \Hd_D(z) \, |\di z|
\end{equation*}
the hyperbolic length of $\gamma$. It is easily checked that
Theorem~\ref{thm:main1} can be restated as follows: \textit{a hyperbolic petal $\Delta$ of a one-parameter semigroup $(\phi_t)$ in $\D$ is conformal if and only if
    \begin{equation} \label{eq:cond0}
      \lim \limits_{T\to +\infty} \Big[ \ell_{\Delta}(\gamma_z|_{[0,T]})-\ell_{\D}(\gamma_z|_{[0,T]}) \Big] ~<~+\infty
      \end{equation}
 for any backward orbit $\gamma_z(t):=\phi_{-t}(z)$, $z \in \Delta$, in the petal $\Delta$\,.} Note that the above limit always exists, finite or infinite, because $\Delta\subset\UD$ and hence ${\Hd_\Delta(z)\ge\Hd_\UD(z)}$ for all~${z\in\Delta}$.

\subsection{The angular derivative problem and hyperbolic length}\label{SUB:relation-to-cond-by-BK}
Theorem~\ref{thm:main1} is closely related to a recent conformality condition due to Betsakos and Karamanlis \cite{BK2022}. We briefly discuss this issue. The conformality conditions obtained in \cite{BK2022} are valid for any simply connected domain $\Omega\subsetneq \C$  containing the real line $\R$. In order to compare the results of \cite{BK2022} with our Theorem~\ref{thm:main2} we additionally assume that $\Omega$ contains the standard strip $\St$. Taking into account Ostrowski's characterization of semi-conformality, see e.g. \cite[Theorem~A]{BK2022}, we see that in this case,
\cite[Theorem~1]{BK2022} says that $\Omega$  is conformal at~$-\infty$ w.r.t.~$\St$ if and only if
\begin{align}
\label{EQ_Ostr-semi-conf}
&\dist(z,\partial\Omega):=\inf_{w\in\partial\Omega}{|z-w|}~\to~0 &&\text{as~}~\Re z\to-\infty,~z\in\partial\St,\\[1.25ex]
\label{EQ:BetsakosKaramanlis}
\text{\hskip-2em and}\qquad &\hdist_{\St} (iy+a,iy+b)-\hdist_{\Omega}(iy+a,iy+b)~\to~0 &&\text{as~}~ a,b \to -\infty,~ a, b \in \R,
\end{align}
for some and hence any $y \in (-\pi/2,\pi/2)$. In fact, in \cite{BK2022} the conformality condition (\ref{EQ:BetsakosKaramanlis}) is stated only for $y=0$, but in our case we additonally have $\St \subset \Omega$, and the proof in \cite{BK2022} also works for $y \in (-\pi/2,\pi/2)$.

If $\Omega$ is the Koenigs domain of a one-parameter semigroup
$(\phi_t)$ and $\St$ is a \textit{maximal} strip contained in~$\Omega$, then condition~\eqref{EQ_Ostr-semi-conf} is automatically satisfied.

On the other hand, it is easy to see that Theorem~\ref{thm:main2} says that $\Omega$ is conformal at~$-\infty$  w.r.t.~$\St$ if and only if for some and hence all~${y \in (-\pi/2,\pi/2)}$,
          \begin{equation} \label{eq:cond1}
     \int \nolimits_{a}^b \big(\Hd_S(iy+x) -\Hd_{\Omega}(iy+x) \big) \, \di x \to 0 \quad \text{ as~}~a,b \to -\infty \, , \, \,  a, b \in \R \, .
  \end{equation}
Let us compare conditions (\ref{EQ:BetsakosKaramanlis}) and (\ref{eq:cond1}).

Since $\int_{a}^b \Hd_S(x) \, \di x=b-a=\hdist_{S}(a,b)$ and $\int_{a}^b \Hd_{\Omega}(x) \, \di x \ge \hdist_{\Omega}(a,b)$   for any $a<b$, $a,b \in \R$,
condition (\ref{EQ:BetsakosKaramanlis}) for ${y=0}$ implies condition (\ref{eq:cond1}) for ${y=0}$.
Thus, if  $\Omega$ is conformal at $-\infty$ w.r.t.~$\St$ and if $\St\subset\Omega$, then regardless of whether $\Omega$ is starlike at infinity or not, our condition~(\ref{eq:cond1}) holds for~${y=0}$.

A similar remark applies to the relation between~(\ref{eq:cond1}) and the following necessary and sufficient condition for conformality established by  Bracci et al~\cite[(8.2)]{Bracci_et_al2019}:
\begin{equation}\label{EQ_Filippos-cond}
 \limsup_{x\to-\infty}\Big[\hdist_\St(iy,\,iy+x) - \hdist_\Omega(iy,\,iy+x)\Big]~<~+\infty
\end{equation}
for some and hence all~${y \in (-\pi/2,\pi/2)}$. Arguing as above, one can see that~\eqref{EQ_Filippos-cond} implies~(\ref{eq:cond1}), but apparently, only for~${y=0}$. Although \cite[Sect.\,8]{Bracci_et_al2019} addresses the angular derivative problem in the context of one-parameter semigroups, the proof of~\eqref{EQ_Filippos-cond} does not depend on the fact that $\Omega$ is starlike at infinity; it actually works for any simply connected domain $\Omega\subsetneq \C$ containing~$\St$.

\subsection{Open questions}\label{SUB:Questions}\mbox{~}

In  conclusion, we state several open questions. Let $(\phi_t)$ be a non-elliptic one-parameter semigroup with associated infinitesimal generator~$G$, Koenigs function~$h$, and Koenigs domain ${\Omega:=h(\UD)}$. Further let $\Delta(\sigma)$ be a hyperbolic petal of~$(\phi_t)$ with $\alpha$-point~$\sigma$. As above, for simplicity we suppose that $h(\Delta(\sigma))=\St$.

\begin{question}
Similarly to a result of Betsakos and Karamanlis~\cite[Theorem~1]{BK2022}, our condition~\eqref{EQ_int-mainthrm1-comvergence}, as well as its restatements \eqref{EQ_int-mainthrm2-comvergence} and~(\ref{eq:cond1}), uses hyperbolic geometry. However, in the proof we make use of euclidean quantities related to the Koenigs domain~$\Omega$. Is it possible to prove one (or even both) of the implications in Theorem~\ref{thm:main1} \textit{without} employing criteria for conformality in terms of euclidean geometry?

\end{question}
One possible way to answer the above question would be to study in detail the relations  between the convergence of the integrals $I(z_0)$ and $J(z_0)$ introduced in Sect.\,\ref{SUB:rate-of-conv}. An alternative direction is indicated in the next question.

\begin{question}
If $\Omega$ is starlike at infinity, then the conformality condition (\ref{EQ:BetsakosKaramanlis}) due to Betsakos and Karamanlis and our condition~(\ref{eq:cond1}) are equivalent in terms of Theorem~\ref{thm:main2} and \cite[Theorem~1]{BK2022}. Is there a more direct way to prove the equivalence (\ref{EQ:BetsakosKaramanlis}) and (\ref{eq:cond1}) for such domains? How are conditions  (\ref{EQ:BetsakosKaramanlis}) and (\ref{eq:cond1})
related, when $\Omega$ is semi-conformal at $-\infty$ in the sense of Ostrowski, see e.g. \cite[Theorem~A]{BK2022}, but not necessarily starlike at infinity?
\end{question}

In the elementary proof of Theorem~\ref{thm:main1} for the symmetric case, given in Sect.\,\ref{SUB:symm-case}, we consider a generic injective holomorphic self-map isogonal at a contact point. In contrast to the general (non-symmetric) case, the argument does not depend on the fact that the image of the self-map is a hyperbolic petal of a one-parameter semigroup in~$\UD$.

On the one hand, the argument in Sect.\,\ref{SUB:symm-case} works only for the backward orbits contained in~$(-1,\tau)$. This is similar to the situation in Sect.\,\ref{SUB:relation-to-cond-by-BK}, where restricting to the symmetry line~${y=0}$ apparently becomes necessary at some point.

On the other hand, motivated by the symmetric case, it is natural to ask whether even in the general case, Theorem~\ref{thm:main1} remains valid for a semigroup of hyperbolic automorphisms $(\phi_t)_{t\ge0}\subset\Aut(\Delta)$ of a simply connected domain $\Delta\subset\UD$, provided that all the backward orbits converge to the point~${\sigma\in\UC}$ at which $\Delta$ is isogonal, but we do not assume that $(\phi_t)$ extends to a one-parameter semigroup of holomorphic self-maps of the whole unit disk~$\UD$. This leads to the following problem.
\begin{question}
Let $f:\UH\to\UH$ be an injective holomorphic self-map with a boundary fixed point at~${\zeta=0}$. Suppose that $f$ is isogonal at~${\zeta=0}$. Is there any relation between the convergence of the integrals \eqref{EQ_I1} and~\eqref{EQ_I2} introduced in Sect.\,\ref{SUB:symm-case} and the finiteness of the angular derivative~$f'(0)$?
\end{question}

A closely related question is as follows.
Let $\varphi$ be a conformal mapping of~$\UD$ onto the hyperbolic petal~$\Delta(\sigma)$ satisfying  ${\varphi(-1)=\sigma}$ and ${\varphi(1)=\tau}$. According to~\cite[Theorems~1 and~3]{ShoiElinZalc}, the mapping~$\varphi$ is a solution to the non-linear ODE
\begin{equation}\label{EQ_El-Sh-Za}
G'(\sigma)(1-z^2)\varphi'(z)=2G(\varphi(z)),\quad z\in\UD.
\end{equation}
\begin{question}
 What kind of non-trivial conclusion about the infinitesimal generator $G$ can be drawn from the equation~\eqref{EQ_El-Sh-Za}, if we suppose that the petal $\Delta(\sigma)$ is conformal?
\end{question}

The next open question concerns the geometry of hyperbolic petals near the Denjoy\,--\,Wolff point $\tau$ of $(\phi_t)$.
\begin{question}
Recall that $\tau\in\partial \Delta(\sigma)$. Let $\tilde\varphi$ be a conformal mapping of~$\UD$ onto $\Delta(\sigma)$ with ${\tilde\varphi(1)=\tau}$. What is the asymptotic behaviour of $\tilde\varphi(z)$ as ${z\to1}$ within a Stolz angle? In particular, if $(\phi_t)$ is hyperbolic, then using results of Contreras and D\'\i{}az-Madrigal~\cite{Analytic-flows} it is possible to show that $\partial \Delta(\sigma)$ has a corner of opening $\alpha\pi$ at~$\tau$, with  ${\alpha:=|G'(\tau)|/G'(\sigma)}$. Is it always true (and if not always, then under which conditions) that the function
\begin{equation*}
 f(z):=\big(\tilde\varphi(z)-\tau\big)^{1/\alpha},\quad z\in\UD,
\end{equation*}
is conformal at~${z=1}$, i.e. $f$ has angular derivative $f'(1)\in\Complex^*:=\C\setminus\{0\}$?
\end{question}

We conclude the paper with a question on parabolic petals. Theorem~\ref{thm:parabolic}, in case of semigroups of zero hyperbolic step, reduces the angular derivative problem for parabolic petals to another problem of similar nature for the Koenigs function~$h$. Although the limit relation~\eqref{EQ:MariaKonstantinos} holds for parabolic petals as well, our Theorem~\ref{thm:main1} does not seem to extend to this case. So it is natural to raise the following question.
\begin{question}
Is it possible to characterize conformality of \textbf{parabolic} petals in terms of the intrinsic hyperbolic geometry of the petal and the backward (or forward) dynamics of the semigroup, without involving the Koenigs function of the semigroup?
\end{question}

\medskip
\noindent\textbf{Acknowledgement.} The authors are very grateful to anonymous referees for careful reading of the manuscript and their valuable comments.

\end{document}